\documentclass[a4paper,12pt]{article}
\usepackage[english]{babel}
\usepackage[utf8x]{inputenc}
\usepackage{amsmath,amsfonts,amssymb,amsthm}
\usepackage{mathtools}
\usepackage{hyperref}
\usepackage{enumitem}
\usepackage{float}
\usepackage{xcolor}
\usepackage{tikz-cd}
\usepackage{titlesec}
\usepackage{graphicx}
\usepackage[all,cmtip]{xy}
\usepackage{geometry}
\usepackage{svg}

\def    \C      {{\mathbb C}}
\def    \R      {{\mathbb R}}
\def \Z {{\mathbb Z}}

\renewcommand{\epsilon}{\varepsilon}

\newtheorem{theorem}{Theorem}[section]
\newtheorem{cor}[theorem]{Corollary}
\newtheorem{lemma}[theorem]{Lemma}
\newtheorem{prop}[theorem]{Proposition}
\newtheorem{rmk}[theorem]{Remark}
\newtheorem{defi}[theorem]{Definition}
\newtheorem*{conj}{Conjecture}
\newtheorem{quest}{Question}

\AtEndDocument{\bigskip{\footnotesize
  \textsc{Universidade Federal do Espírito Santo, Vitória-ES, Brazil} \par  
  \textit{E-mail}:  \texttt{brayan.ferreira@ufes.br or brayan.ferre@hotmail.com}
  }}

\title{Elliptic Reeb orbit on some real projective three-spaces via ECH}
\date{}
\author{Brayan Ferreira}

\begin{document}
\maketitle

\begin{abstract}
We prove the existence of an elliptic Reeb orbit for some contact forms on the real projective three space $\R P^3$. The main ingredient of the proof is the existence of a distinguished pseudoholomorphic curve in the symplectization given by the $U$ map on ECH. Also, we check that the first value on the ECH spectrum coincides with the smallest action of null-homologous orbit sets for $1/4$-pinched Riemannian metrics and compute the ECH spectrum for the irrational Katok metric example.
\end{abstract}

\section{Introduction}
Given a $(2n-1)$ dimensional closed oriented manifold $Y$ equipped with a contact form $\lambda$, i.e., $\lambda \wedge d\lambda^{(n-1)}>0$, the \emph{Reeb vector field} $R$ is defined implicitly by the equations $d\lambda (R,\cdot) =0$ and $\lambda(R) \equiv 1$ on $Y$. The flow $\phi_t$ induced by $R$ is then called the \emph{Reeb flow} and a closed trajectory $\gamma \colon \R /T\Z \to Y$ for the Reeb flow is called a \emph{Reeb orbit}. There is a famous conjecture due to Weinstein \cite{WEINSTEIN1979353} asserting that every contact manifold $(Y,\lambda)$ admits a Reeb orbit. Although the Weinstein conjecture is still open in full generality, it has been proved in some cases, the most general one is the positive answer due to Taubes in dimension $3$ \cite{taubes2007seiberg}. In this paper, our focus will be on the case where $n=2$.

We denote by $\xi =\ker \lambda \subset TY$ the two plane distribution defined by $\lambda$, namely the \emph{contact structure}. For a Reeb orbit $\gamma \colon \R /T\Z \to Y$, one defines the linearized Poincaré map $P_\gamma := d\phi_T|_\xi \colon \xi_{\gamma(0)} \to \xi_{\gamma(0)}$. Since $P_\gamma$ is a symplectic linear map, its eigenvalues are inverse to each other. We say that $\gamma$ is \emph{elliptic} if the eigenvalues of $P_\gamma$ are norm one complex numbers and \emph{irrationally elliptic} if their arguments as complex numbers are irrational. The Reeb orbit $\gamma$ is \emph{positive (resp. negative) hyperbolic} if these eigenvalues are positive (resp. negative) real numbers. Moreover, $\gamma$ is \emph{nondegenerate} if $P_\gamma$ does not admit $1$ as an eigenvalue and $\lambda$ is a nondegenerate contact form if every Reeb orbit on $(Y,\lambda)$ is nondegenerate.

The Embedded Contact Homology (ECH) is an algebraic invariant of closed contact $3$-manifolds in which has shown to be very useful to understand symplectic embeddings in dimension $4$ and Reeb dynamics in dimension $3$. In a nutshell, ECH is a homology generated by suitable sets of Reeb orbits on $(Y,\lambda)$ such that differential counts some punctured pseudoholomorphic curves in the symplectization $(\R \times Y, d(e^s\lambda))$ which asymptotically converges (near the punctures) to the chain complex generators, as reviewed in Section \ref{secech}.  In fact, it is a fundamental tool in the proof of $3d$ Weinstein conjecture and some of its refinements, see e.g. \cite{hutchings2010taubes,cristofaro2016one,cristofaro2019torsion}.  In \cite{cristofaro2019torsion}, Cristofaro-Gardiner, Hutchings and Pomerleano used the ECH structure and its $U$ map to prove the existence of global surfaces of section and, applying a result due to Franks \cite[Theorem 4.4]{franks1996area}, they proved the following improvement of $3d$ Weinstein conjecture.

\begin{theorem}[Theorem 1.4 in \cite{cristofaro2019torsion}]\label{2orinfty}
Let $Y$ be a closed connected three-manifold and let $\lambda$ be a nondegenerate contact form on $Y$ . Assume that $c_1(\xi)\in H^2(Y;\Z)$ is torsion. Then $\lambda$ has either two or infinitely many simple Reeb orbits.
\end{theorem}

This result was extended dropping the condition on the first Chern class $c_1(\xi)$ by Colin, Dehornoy and Rechtman in \cite[Theorem 1.1+Corollary 4.8]{colin2022existence} and, more recently, Cristofaro-Gardiner, Hryniewicz, Hutchings and Liu proved that the nondegeneracy condition in Theorem \ref{2orinfty} is not necessary, i.e., if $c_1(\xi)$ is torsion, there must exist two or infinitely many simple Reeb orbits on $(Y,\lambda)$, see \cite[Theorem 1.1]{cristofaro2023proof}. Moreover, in a previous work also using ECH tools, they completely described the case of a contact form admitting exactly two Reeb orbits. 

\begin{theorem}[Theorem 1.2 in \cite{cristofaro2021contact}]\label{2orbits}
Let $Y$ be a closed three-manifold, and let $\lambda$ be a contact form on $Y$ with exactly two simple Reeb orbits. Then $\lambda$ is nondegenerate and both Reeb orbits are irrationally elliptic. Furthermore, if $Y$ is a lens space, then $\lambda$ is dynamically convex, $\xi$ is tight and there is a direct relation between the contact volume of $Y$ and the periods of the simple Reeb orbits.
\end{theorem}

Inspired by these results, one may ask for even more specific qualitative properties of Reeb flows on a given manifold. In this paper, we apply ECH tools to study the qualitative properties of Reeb flows on the real projective 3-space, $\R P^3$. Namely, we study a refinement of the Weinstein conjecture for this particular manifold, proving the existence of an elliptic Reeb orbit under some assumptions.

\subsection{Elliptic Reeb orbit on $\R P^3$}
Before stating the main results of this paper, we introduce some notation. Recall that given a Reeb orbit $\gamma$ and a symplectic trivialization $\tau$ of $\xi|_\gamma$, there is a well defined Conley--Zehnder index $CZ_\tau(\gamma)$. This index is a well defined integer depending just on the trivialization $\tau$ when the first Chern class $c_1(\xi)$ vanishes, and has a simple description for nondegenerate Reeb orbits in dimension $3$, as we now recall. Let $\gamma$ be a nondegenerate Reeb orbit on $(Y,\lambda)$. If $\gamma$ is hyperbolic, the linearized Reeb flow rotates an eigenvector of the Poincaré map $P_\gamma$ by angle $\pi k$, for some integer\footnote{The integer $k$ is even when $\gamma$ is positive hyperbolic and odd in the negative hyperbolic case.} $k$, and 
$$CZ_\tau (\gamma^n) = n k.$$ Here $\gamma^n$ denotes the $n$-fold iterate
\begin{eqnarray*}
\gamma^n\colon \R/T\Z &\to& Y \\ s &\mapsto& \gamma(ns)
\end{eqnarray*} of the Reeb orbit $\gamma$. In particular, when $\gamma$ is a hyperbolic Reeb orbit, the Conley--Zehnder index $CZ_\tau(\gamma)$ is linear with respect to the iterates of $\gamma$. On the other hand, if $\gamma\colon \R/T\Z \to Y$ is elliptic, the linearized Reeb flow $d\phi_t|_\xi$ is conjugate to a rotation by angle $2\pi \theta_t \in \R$, where $\theta_t$ is continuous with respect to $t \in [0,T]$ and $\theta_0 = 0$. In this case, one has
$$CZ_\tau(\gamma^n) = 2\lfloor n \theta \rfloor +1,$$
where $\theta = \theta_T$ is the \emph{rotation angle} of $\gamma$ with respect to $\tau$.

We denote by $CZ(\gamma)$ the Conley--Zehnder index with respect to a symplectic trivialization that extends over a disk bounded by $\gamma$. In particular, we use it for a global trivialization of $\xi$ (in case of a trivial bundle).

\begin{defi}
Let $(Y,\lambda)$ be a three dimensional contact manifold such that $c_1(\ker \lambda)|_{\pi_2(Y)} = 0$. The contact form $\lambda$ is called linearly positive if $CZ(\gamma)>0$ for every contractible Reeb orbit. Moreover, $\lambda$ is dynamically convex if $CZ(\gamma) \geq 3$ for every contractible Reeb orbit $\gamma$.
\end{defi}

We recall that $\R P^3 = L(2,1)$ admits a unique tight contact structure up to isotopy, see \cite[Theorem 2.1]{MR1786111}, and the standard tight contact structure $\xi_0 = \ker \lambda_0$ is a trivial symplectic vector bundle. Here $\lambda_0$ denotes the induced contact form on $\R P^3$ by the restriction to the three sphere $S^3$ of the standard Liouville form
$$\lambda_0 = \frac{1}{2}(x_1 dy_1 - y_1 dx_1 + x_2 dy_2 - y_2 dx_2)$$
defined on $\R ^4$. Now we are ready to state the first result of this paper.

\begin{theorem}\label{contractibleind2}
Let $\lambda$ be a nondegenerate linearly positive contact form on $\R P^3$ defining a tight contact structure $\xi$. Suppose $\lambda$ does not admit a contractible Reeb orbit with Conley--Zehnder index $2$. Then, the Reeb flow for $\lambda$ has an elliptic Reeb orbit with Conley--Zehnder index $1$. In particular, this holds when $\lambda$ is nondegenerate and dynamically convex.
\end{theorem}

\begin{rmk} Leonardo Macarini pointed out to the author that this result also follows from $S^1$-equivariant symplectic homology theory. In fact, this homology is generated by Reeb orbits and it admits a grading given by the Conley--Zehnder index. Since the degree $1$ group is nontrivial, there must exist an orbit with Conley--Zehnder index $1$. The fact that this orbit is elliptic follows from the hypothesis on Conley--Zehnder index $2$ orbits and the behavior of the index under iterations. In addition, the dynamically convex case in the theorem also follows from a more general result due to Hryniewicz and Salomão in \cite{hryniewicz2016elliptic}. However, the approach we shall follow in the proof here, despite being related, is different from the one followed by them. Further, still in this case, there is a more general result for (possibly degenerate) dynamically convex contact forms in $\R P^{2n+1}$ due to Leonardo Macarini and Miguel Abreu, see \cite[Corollary 2.7]{abreu2017dynamical}.
\end{rmk}

In \cite{hutchings2011quantitative}, Hutchings defined the ECH spectrum for a contact three manifold $(Y,\lambda)$. This is a sequence of nonnegative numbers
$$ 0 = c_0(Y,\lambda) < c_1(Y,\lambda) \leq c_2(Y,\lambda)\leq \ldots \leq \infty, $$
defined using the $U$ map on ECH. These numbers have nice properties that we shall not discuss here except for one: if $c_k(Y,\lambda)<\infty$, then there exists an orbit set $\alpha$, which is null-homologous, with $c_k(Y,\lambda) = \mathcal{A}(\alpha)$. Here an orbit set is a set of the form $\alpha = \{(\gamma_i,m_i)\}$, where $\gamma_i$ are embedded Reeb orbits and $m_i$ are positive integers, and $\mathcal{A}(\alpha)$ denotes the \emph{total action} of $\alpha$, i.e., $$\mathcal{A}(\alpha) = \sum_i m_i \mathcal{A}(\gamma_i) := \sum_i m_i \int_{\gamma_i} \lambda.$$
Moreover, $\alpha$ being a \emph{null-homologous orbit set} means that the total homology class
$$[\alpha] := \sum_i m_i [\gamma_i] \in H_1(Y;\Z)$$
is equal to zero. The described property is sometimes called the \emph{spectrality property}. Denote by $\mathcal{A}^0_{\min}(\lambda)$ the smallest action of a null-homologous orbit set. By the spectrality property, it is easy to see that $c_1(Y,\lambda)\geq \mathcal{A}^0_{\min}(\lambda)$. Under the hypothesis that $c_1(Y,\lambda)$ is realized by $\mathcal{A}^0_{\min}(\lambda)$, we can weaken the condition on Conley--Zehnder index $2$ orbits in Theorem \ref{contractibleind2} and obtain an estimate for the action of the elliptic Reeb orbit found.

\begin{theorem}\label{embeddedind2}
Let $\lambda$ be a nondegenerate linearly positive contact form on $\R P^3$ defining a tight contact structure $\xi$. Suppose that every contractible Reeb orbit with Conley--Zehnder index $2$ is embedded. Moreover, suppose that the first value of the ECH spectrum $c_1(\R P^3,\lambda)$ is equal to the smallest action of a null-homologous orbit set $\mathcal{A}^0_{\min}(\lambda)$. Then, the Reeb flow for $\lambda$ has an elliptic Reeb orbit with Conley--Zehnder index $1$ and action in the interval $[\mathcal{A}^0_{\min}(\lambda)/2,\mathcal{A}^0_{\min}(\lambda)]$.
\end{theorem}

In fact, the arguments presented in the proofs of Theorem \ref{contractibleind2} and Theorem \ref{embeddedind2} are almost the same. In particular, one concludes that the elliptic Reeb orbit found in Theorem \ref{contractibleind2} has action in the interval $[c_1(\R P^3,\lambda)/2,c_1(\R P^3,\lambda)]$.

We note that since $H_1(\R P^3;\Z) = \Z_2$, the inequality $\mathcal{A}_{min}^0(\lambda) \leq 2\mathcal{A}_{min}(\lambda)$ must hold, where $\mathcal{A}_{min}(\lambda)$ is the minimal action between all the Reeb orbits for $\lambda$. Therefore, the elliptic Reeb orbit obtained in the previous theorem has action in the interval
$$[\mathcal{A}_{min}(\lambda), \mathcal{A}^0_{min}(\lambda)] \subset [\mathcal{A}_{min}(\lambda),2\mathcal{A}_{min}(\lambda)].$$

Similar results can be obtained analogously to what we do for some tight lens spaces. In fact, after the initial version of this paper, Shibata has independently proved interesting related results for dynamically convex contact forms on lens spaces $L(p,p-1)$, see \cite{shibata2023dynamically}.

Since the case of a contact form admitting exactly two Reeb orbits was described by Theorem \ref{2orbits}, the interesting cases for Theorem \ref{contractibleind2} and Theorem \ref{embeddedind2} are those with infinitely many Reeb orbits. In this case, we have the following result due to Shibata.

\begin{theorem}(Theorem 1.6 in \cite{shibata2022existence})
Let $Y$ be a closed three manifold such that $b_1(Y) = rk(H_1(Y;\mathbb{Z})) = 0$. Suppose that $\lambda$ is a nondegenerate contact form on $Y$ such that the Reeb flow has infinitely many simple periodic orbits and at least one elliptic orbit. Then, there exists at least one simple positive hyperbolic orbit.
\end{theorem}

Putting this together Theorems \ref{2orbits}, \ref{contractibleind2} and \ref{embeddedind2}, we obtain the following consequence.

\begin{cor}
Let $\lambda$ be a contact form on $\R P^3$ satisfying the hypotheses of Theorem \ref{contractibleind2} or Theorem \ref{embeddedind2}. Then, either $\lambda$ admits exactly two Reeb orbits being both of them irrationally elliptic or $\lambda$ admits infinitely many simple Reeb orbits with at least one being elliptic with Conley--Zehnder index $1$ and at least one being positive hyperbolic.
\end{cor}

\begin{rmk}\label{hypinfi} It follows from Theorem \ref{2orinfty} and Theorem \ref{2orbits} that the property of having infinitely many simple Reeb orbits is a necessary condition to the existence of a hyperbolic Reeb orbit for any nondegenerate contact form on\footnote{In fact, it holds on any closed three dimensional manifold due to the ``$2$ or infinitely many orbits'' result due to Colin, Dehornoy and Rechtman in \cite{colin2022existence}.} $\R P^3$.
\end{rmk}
 
\subsection{Elliptic closed geodesic on Finsler $2$-spheres}
Let $F\colon TS^2 \to [0,\infty)$ be a Finsler metric\footnote{For the definition of a Finsler metric, see Section \ref{finsler}.} on the two sphere $S^2$. We denote by $S_FS^2 = F^{-1}(1)$ the unit tangent bundle for the Finsler metric $F$. As we shall recall in Section \ref{finsler}, $S_FS^2 \cong \R P^3$ admits a contact form, namely the Hilbert form $\lambda_F$, such that the Reeb flow coincides with the Geodesic flow for $F$. In particular, a closed geodesic on $(S^2,F)$ has the same type (elliptic or hyperbolic) as a corresponding Reeb orbit on $(S_FS^2,\lambda_F)$. Moreover, the contact structure $\xi_F = \ker \lambda_F$ is tight and symplectically trivial. Further, the Conley--Zehnder index of a Reeb orbit with respect to a global trivialization of $\xi_F$ agrees with the Morse index of the corresponding geodesic on $S^2$.

There are still seemingly simple open problems about closed geodesics for a given Finsler metric $F$ on $S^2$. For instance, it is known that for any reversible $F$ there exists infinitely many closed geodesics but it is still a conjecture that there must be two or infinitely many closed geodesics for an irreversible $F$, see e.g. \cite[Conjecture 1]{long2006multiplicity}. Theorem \ref{2orinfty} above gives a positive answer to this conjecture for the bumpy\footnote{A metric $F$ is bumpy when every closed geodesic is nondegenerate.} case. In addition, Remark \ref{hypinfi} confirms a conjecture due to Long asserting that the existence of a hyperbolic prime closed geodesic on a Finsler $S^2$ implies the existence of infinitely many prime closed geodesics, see \cite[Conjecture 2.2.2]{burns2021open}, but also just for the bumpy case. Moreover, there is another conjecture by Long directly related to the two previous results.

\begin{conj}[Conjecture 5 in \cite{long2006multiplicity}]
There exists at least one elliptic closed geodesic for any Finsler metric on the two sphere $S^2$.
\end{conj}

As already observed in \cite{hryniewicz2016elliptic}, a nice consequence of Theorem \ref{contractibleind2} related to this conjecture is the following result. Let $r:=\max \{F(-v)\mid F(v)=1 \}\geq 1$ be the reversibility of the Finsler metric $F$, as defined by Rademacher.

\begin{cor}
Let $(S^2,F)$ be a bumpy Finsler sphere with reversibility $r$. If $F$ is $\left( \frac{r}{1+r}\right)^2$-pinched, i.e., 
\begin{equation}\label{dynpinched}
\left( \frac{r}{1+r}\right)^2 < K \leq 1,
\end{equation}
for all flag curvatures $K$, then there exists an elliptic closed geodesic with Morse index $1$ on $(S^2,F)$.
\begin{proof}
Harris and Paternain proved in \cite{harris2008dynamically} using a length estimate of the shortest geodesic loop for $F$ due to Rademacher in \cite{RADEMACHER2008763}, that condition \eqref{dynpinched} is sufficient for the Hilbert form $\lambda_F$ being dynamically convex on $S_FS^2 \cong \R P^3$. Hence, this Corollary follows from Theorem \ref{contractibleind2}.
\end{proof}
\end{cor}

For a more general and detailed discussion containing this corollary, see \cite[Section 2.3]{abreu2017dynamical}, where they extend the existence of an elliptic closed geodesic even for the case that one does not have the strict inequality in \eqref{dynpinched} and obtain other similar results.

\subsection{Computations on ECH spectra}
Inspired in Theorem \ref{embeddedind2}, one may ask the following question.

\begin{quest}\label{whichctc}
For which contact forms on $\R P^3$ does $c_1(\R P^3,\lambda) = \mathcal{A}_{min}^0(\lambda)$ hold?
\end{quest}

Using a result due to Ballman, Thorbergsson and Ziller \cite[Theorem 4.2]{MR727711} we give the following partial answer to Question \ref{whichctc}.

\begin{theorem}\label{thmc1}
Let $(S^2,g)$ be a Riemannian sphere such that $1/4 < K \leq 1$, where $K$ is the sectional curvature. Then
\begin{equation}\label{c1=action}
c_1(S_gS^2,\lambda_g) = \mathcal{A}_{min}^0(\lambda_g) = 2L,
\end{equation}
where $L$ is the length of a shortest closed geodesic for $g$. Moreover, it is well known that $L \in [2\pi,4\pi)$ in this case.
\end{theorem}

In \cite[Proposition 3.2]{ferreira2023gromov} one can check that property \eqref{c1=action} holds for any Zoll metric on the sphere. Moreover, from \cite[Proposition 1.9]{ferreira2023gromov} we note that it also holds for metrics corresponding to ellipsoids of revolution in $\R^3$. There are also examples coming from quotients of suitable symmetric hypersurfaces in $\C^2$ called \emph{monotone toric domains}, see e.g. \cite[Theorem 1.7]{gutt2022examples}.

The author believes that Theorem \ref{thmc1} might hold for Finsler metrics satisfying the pinching condition \eqref{dynpinched} as well. The proof we present here uses the reversibility of the Riemannian metric and the existence of the Birkhoff annulus for an embedded closed geodesic in positive curvature. Nevertheless, it might be adapted to that more general case as long as one can obtain versions of Klingenberg \cite[Theorem 2.6.9]{klingenberg1995riemannian} and Toponogov \cite[Theorem 2.7.12]{klingenberg1995riemannian} estimates, and a result due to Ballman, Thorbergsson and Ziller \cite[Theorem 4.2]{MR727711} that we use, for nonreversible Finsler metrics.

The positivity of the curvature is necessary to ensure the equality $c_1(S_gS^2,\lambda_g)= \mathcal{A}_{min}^0(\lambda_g)$ in Theorem \ref{thmc1}. To see this, consider the \emph{dumbbell metric} $g$ on $S^2$, that is, a metric isometric to the dumbbell surface in the Euclidean space $\R^3$. In this dumbbell, each half is close to the round sphere of constant curvature $K=1$ and the pipe connecting the two halves has a sufficiently small radius $\varepsilon>0$ such that the shortest closed geodesic has length $2\pi \varepsilon$, see e.g. \cite[Figure 1]{Cheeger+1971+195+200} or \cite[Figure 4]{calabi1992simple}. In this case, one has
$$\mathcal{A}^0_{min}(\lambda_g) = 4\pi \varepsilon < 2\pi \leq c_1(S_gS^2,\lambda_g).$$
The inequality $c_1(S_gS^2,\lambda_g)\geq 2\pi$ can be verified using the monotonicity property of the ECH spectrum and the existence of a symplectic embedding
$(\mathrm{int}B(2\pi),\omega_0) \hookrightarrow (\mathrm{int}D^*_{g}S^2,\omega_{can})$. Here, 
$$B(2\pi) = \left\{(x_1,x_2,y_1,y_2) \in \R^4 \mid \sum_{j=1}^2 x_j^2 + y_j^2 \leq 2\right\}$$
denotes the Euclidean ball of capacity $2\pi$, $\omega_0 = \sum_{j=1}^2 dx_j \wedge dy_j$ is the standard symplectic form on $\R^4$, $D^*_gS^2 = \{(q,p) \in T^*S^2 \mid \Vert p \Vert_g \leq 1 \}$ denotes the unit disk cotangent bundle over $S^2$ with respect to the metric $g$ and $\omega_{can}$ is the canonical symplectic form on the cotangent bundle $T^*S^2$, locally given by $\sum_{j=1}^2 dp_j \wedge dq_j$ in cotangent coordinates. The existence of such a symplectic embedding follows from \cite[Theorem 1.3]{ferreira2021symplectic}. This embedding yields the inequality
$$2\pi = c_1(\partial B(2\pi),\lambda_0) \leq c_1(S_gS^2,\lambda_g),$$
where $\lambda_0$ is once again the restriction of the standard Liouville form on $\R^4$.

\begin{rmk}
Note that by \cite[Corollary 6.1]{harris2008dynamically}, we have that $\lambda_g$ is dynamically convex for a $1/4$-pinched Riemannian metric $g$ on the sphere $S^2$. Therefore, inspired by Theorem \ref{thmc1}, the author conjectures that $c_1(\R P^3,\lambda) = \mathcal{A}_{min}^0(\lambda)$ holds for every dynamically convex contact form $\lambda$. On the other hand, this equality holds for Hilbert forms coming from ellipsoids of revolution in $\R^3$ but, as already noted by Harris and Paternain in \cite[$\S 6$]{harris2008dynamically}, some of them are not dynamically convex contact forms on $\R P^3$.
\end{rmk}

Katok found interesting examples of Finsler metrics in \cite{katok1973ergodic}. Among them, he studied a family of irreversible metrics with only finitely many closed geodesics. In particular, the irrational Katok metric example on $S^2$ gives a irreversible metric $F_a$ admitting exactly two closed geodesics for every irrational number $a \in (0,1)$. The last result of this paper is the computation of the ECH spectrum for this example. For this, we recall that, given two real numbers $a,b>0$, one defines the sequence $N(a,b)$ consisting of all nonnegative integer linear combinations of $a$ and $b$ arranged in nondecreasing order, and indexed starting at $0$. We denote by $M_2(N(a,b))$ the subsequence of $N(a,b)$ formed by the linear integer combinations with even total weight, i.e., combinations $ma + nb$ such that $m+n$ is an even nonnegative integer.

\begin{theorem}\label{katokspectrum}
Let $a \in (0,1)$ be an irrational number and $F_a$ be the Katok metric for the two sphere. The ECH spectrum of its unit tangent bundle equipped with the Hilbert form is given by
$$(c_k(S_{F_a}S^2,\lambda_{F_a}))_k = M_2\left (N \left (\frac{2\pi}{1+a},\frac{2\pi}{1-a} \right ) \right ).$$
More precisely, $c_k(S_{F_{a}}S^2,\lambda_{F_a})$ is the $k$-th term in the sequence of nonnegative integer combinations $m_1\frac{2\pi}{1+a} + m_2 \frac{2\pi}{1-a}$ such that $m_1+m_2$ is even, ordered in nondecreasing order. In particular,
$$c_1(S_{F_a}S^2,\lambda_{F_a}) = \frac{4\pi}{1+a} = \mathcal{A}_{min}^0(\lambda_{F_a})= 2L.$$
\end{theorem}

We note that, by continuity of the ECH spectrum with respect to the contact form, the theorem above holds for any $a \in [0,1)$ and, for $a=0$, this recovers the computation of the ECH spectrum for the round metric obtained in \cite[Theorem 1.4]{ferreira2021symplectic}.

\begin{rmk}
In \cite[Lemma 4.1]{harris2008dynamically} it is shown that there is a double covering $S^3 \to S_FS^2$ for any Finsler metric on the sphere $S^2$. In the case of the Katok metric $F_a$, it yields a double covering $\phi \colon S^3 \to S_{F_a}S^2$ such that $\phi^*\lambda_{F_a} = 2 h\lambda_0$, where $h\lambda_0$ is the contact form corresponding to the ellipsoid
$$\partial E\left(\frac{2\pi}{1+a},\frac{2\pi}{1-a}\right) = \left\{(z_1,z_2) \in \C^2 \mid \pi\left(\frac{1+a}{2\pi}\vert z_1 \vert^2 + \frac{1-a}{2\pi}\vert z_2 \vert^2\right) = 1\right\},$$
Moreover, Hutchings computed the ECH spectrum of the ellipsoid $\partial E(a,b)$ in \cite{hutchings2011quantitative}, obtaining
$$c_k(\partial E(a,b),\lambda_0) = N(a,b)_k,$$
for any $a,b>0$. In particular, for $a \in [0,1)$ the ECH spectrum $c_k(S_{F_a},\lambda_{F_a}) = M_2\left (N \left (\frac{2\pi}{1+a},\frac{2\pi}{1-a} \right ) \right )_k$ is a distinguished subsequence of the ECH spectrum of the corresponding ellipsoid.
\end{rmk}

As in \cite[\S 4, \S 5]{ferreira2021symplectic}, the computation of the ECH and the ECH spectrum for the Katok example found in this paper, suggests an interesting relation and a method for computing ECH elements of global quotients of hypersurfaces in $\C^2$. This relation probably can be better explored in a more general context in future works.

\vskip 8pt

\noindent{{\bf Acknowledgments:}}
The author would like to thank Vinicius Ramos, Umberto Hryniewicz and Leonardo Macarini for helpful conversations. Thanks also to Marco Mazzucchelli and Lucas Ambrozio for pointing out some facts that helped in the proof of the degenerate case in Theorem \ref{thmc1}. Additionally, the author appreciates the anonymous referees for the careful reading and interesting suggestions.

\section{Quick review on ECH}\label{secech}
We start with a quick review on the theory of embedded contact homology. For a more detailed explanation on this subject, we recommend \cite{hutchings2014lecture}.

\subsection{Chain complex}
Let $Y$ be a three dimensional closed manifold and $\lambda$ be a nondegenerate contact form on $Y$. For a fixed $\Gamma \in H_1(Y;\Z)$, the ECH chain complex $ECC_*(Y,\lambda,\Gamma)$ is the $\Z_2$-vector space generated by \emph{admissible orbit sets}, i.e., sets of the form $\alpha = \{(\alpha_i,m_i)\}$, where
\begin{itemize}
\item $\alpha_i \colon \R/T_i\Z \to Y$ are embedded Reeb orbits.
\item $m_i$ are positive integers satisfying $m_i = 1$, whenever $\alpha_i$ is (positive or negative) hyperbolic.
\item $[\alpha] = \sum_i m_i[\alpha_i] = \Gamma \in H_1(Y;\Z)$.
\end{itemize}
We often use the product notation $\alpha = \Pi_i \gamma_i^{m_i}$, and we commonly refer to an admissible orbit set simply as an \emph{ECH generator}. The differential $\partial$ is defined using a generic \emph{symplectization-admissible} almost complex structure $J$ on the symplectization $\R \times Y$ meaning that $J$ satisfies:
\begin{itemize}
\item $J \frac{\partial}{\partial s} = R$, where $R$ is the Reeb vector field defined by $\lambda$ and $s$ is the coordinate on $\R$.
\item $J\xi = \xi$, where $\xi = \ker \lambda$ is the contact structure.
\item $d\lambda(v,J v)> 0$ for every nonzero vector $v \in \xi$.
\end{itemize}
For two ECH generators $\alpha, \beta$, the coefficient $\langle \partial \alpha, \beta \rangle$ is defined to be a $\Z_2$-count of $J$-holomorphic currents on the symplectization $\R \times Y$ with \emph{ECH index} $1$ and which converge as currents to $\alpha$ (resp. to $\beta)$ when $s$ tends to $+\infty$ (resp. to $-\infty$). We note that since $\langle \partial \alpha, \beta \rangle \neq 0$ implies the existence of a $J$-holomorphic current connecting $\alpha$ to $\beta$, it must hold $\mathcal{A}(\alpha) > \mathcal{A}(\beta)$, i.e., the differential decreases the action. In fact, if $J$ is symplectization-admissible, then $d\lambda$ is pointwise nonnegative restricted to any $J$-holomorphic curve in $\R \times Y$ and the inequality follows from Stoke's theorem. This inequality is strict because $d\lambda(w,Jw) \geq 0$ for every $w \in T(\R \times Y)$ and $d\lambda(w,Jw) = 0$ if, and only if, $w$ is in the subspace generated by $\frac{\partial}{\partial s}$ and the Reeb vector field $R$. Hence, if $\widetilde{u} \colon \Sigma \to \R \times Y$ is a connected $J$-holomorphic curve such that $\int_{\Sigma} \widetilde{u}^*d\lambda = 0$, $\widetilde{u}$ must be a \emph{trivial cylinder} $\R \times \gamma$, where $\gamma$ is a Reeb orbit on $(Y,\lambda)$. More precisely, we must have
\begin{eqnarray*}
\widetilde{u} \colon \R \times \R/T\Z &\to& \R \times Y \\ (s,t) &\mapsto& (s,\gamma(t)).
\end{eqnarray*}

\subsection{ECH index and grading}\label{sectgrading}
Given two orbit sets $\alpha = \Pi_i \alpha_i^{m_i}$ and $\beta = \Pi_j \beta_j^{n_j}$, we denote by $H_2(Y,\alpha,\beta)$ the affine space over the singular homology group $H_2(Y)$ consisting of $2$-chains $\Sigma$ in $Y$ such that
$$\partial \Sigma = \sum_i m_i\alpha_i - \sum_j n_j \beta_j.$$
Given a homology class $Z$ in $H_2(Y,\alpha,\beta)$, its ECH index is defined by the equation
$$I(Z) = c_\tau(Z) + Q_\tau(Z) + CZ_\tau^I (\alpha) - CZ_\tau^I (\beta),$$
where $\tau$ is a trivialization of the contact structure $\xi$ over the orbits appearing in $\alpha$ and $\beta$, $c_\tau(Z)$ is the \emph{relative Chern class}, $Q_\tau(Z)$ is a relative intersection number and $CZ_\tau^I(\alpha) = \sum_i \sum_{k=1}^{m_i} CZ_\tau(\alpha_i^k)$ is a sum of Conley--Zehnder indices of iterates of the orbits $\alpha_i$, and similarly for $CZ_\tau^I (\beta)$. More precisely, given a smooth map $f\colon S \to Y$ representing $Z$, $c_\tau(Z) = c_1(\xi|_{f(S)},\tau)$ is the number of zeros of a generic section of $f^*\xi$ obtained by extending a nonvanishing section of $f^*\xi|_{\partial S}$.

The \emph{relative intersection number} $Q_\tau(Z)$ is defined as $Q_\tau(Z,Z)$ where $Q_\tau(Z,Z^\prime)$ denotes the signed count of transverse intersections for suitable representatives of $Z$ and $Z^\prime$ in $(-1,1) \times Y$. We note that this number is quadratic on $Z$:
$$Q_\tau(Z+Z^\prime) = Q_\tau(Z) + 2 Q_\tau(Z,Z^\prime) + Q_\tau(Z^\prime).$$
In addition, given two null-homologous Reeb orbits $\gamma_1,\gamma_2$ on $(Y,\lambda)$, for two classes $Z_1 \in H_2(Y,\gamma_1, \emptyset)$ and $Z_2 \in H_2(Y,\gamma_2, \emptyset)$, we have $Q_\tau(Z_1,Z_2) = lk(\gamma_1,\gamma_2)$. Here $lk(\gamma_1,\gamma_2)$ denotes the \emph{linking number}
$$lk(\gamma_1, \gamma_2) := \gamma_1 \cdot S_{\gamma_2} \in \mathbb{Z},$$
where $S_{\gamma_2}$ is an embedded Seifert surface for the null-homologous oriented knot $\gamma_2$ which is transverse to $\gamma_1$. Moreover, given a knot $\gamma$ which is transverse to the contact structure $\xi$, e.g., a Reeb orbit, one defines the \emph{self-linking number}:
$$sl(\gamma, S_\gamma) = lk(\gamma, \gamma^\prime),$$
where $S_\gamma$ is a Seifert surface for $\gamma$ and $\gamma^\prime$ is a parallel copy of $\gamma$ obtained pushing $\gamma$ in the direction of a nonvanishing section of $\xi|_{S_\gamma}$. We note that this number does not depend on $S_\gamma$ when $c_1(\xi) \in H^2(Y;\Z)$ vanishes, and hence, we write $sl(\gamma)$. Following the definitions, one obtains the relation
$$sl(\gamma) = Q_\tau(Z) - c_\tau(Z),$$
whenever $Z \in H_2(Y,\gamma,\emptyset)$ and $\gamma$ is a null-homologous Reeb orbit.

With substantial work, Hutchings and Taubes verified that the differential is well defined and satisfies $\partial ^2 = \partial \ \circ \ \partial = 0$, see \cite{hutchings2007gluing,hutchings2009gluing}. The resulting homology is denoted by $ECH_*(Y,\lambda,\Gamma,J)$. Further, one can define a (relative) grading in the chain complex $ECC_*(Y,\lambda,\Gamma,J)$ in the following way. Fix an admissible orbit set $\beta$ such that $[\beta] = \Gamma$ and put $\vert \beta \vert = 0$. For any other admissible $\alpha$ in the class $\Gamma$, we define
\begin{equation}\label{gradingdef}
\vert \alpha \vert := I(Z),
\end{equation}
where $Z \in H_2(Y,\alpha,\beta)$ is an arbitrary class. By Index ambiguity formula in \cite[\S 3.4]{hutchings2014lecture}, if we choose another class $Z^\prime$ in $H_2(Y,\alpha,\beta)$,
$$I(Z) - I(Z^\prime) = \langle c_1(\xi)+2PD(\Gamma), Z-Z^\prime \rangle$$
holds. Hence, \eqref{gradingdef} is not a well defined integer. Nevertheless, it has a well defined class in $\Z_d$, where $d$ is the integer such that the subgroup
$$\{\langle c_1(\xi) + 2 PD(\Gamma), h\rangle \mid h \in H_2(Y;\Z)\} \subset \Z$$ is isomorphic to $d\Z$. In particular, if $c_1(\xi) + 2 PD(\Gamma) = 0$, equation \eqref{gradingdef} defines an actual $\Z$-grading. Moreover, for the singular homology class $\Gamma = 0$, one has the distinguished choice of picking $\beta$ as the empty set $\emptyset$. The ECH index has an additivity property which ensures that $\vert \sigma \vert \equiv \vert \alpha \vert - 1 \mod d$, for every $\sigma$ such that $\langle \partial \alpha, \sigma \rangle \neq 0$.

\subsection{U map and ECH spectrum}
There is also a degree $-2$ map
$$U \colon ECH_*(Y,\lambda,\Gamma,J) \to ECH_{*-2}(Y,\lambda,\Gamma,J),$$
coming from a map defined in the chain complex. Similarly to the differential,  for two ECH generators $\alpha, \beta$, the coefficient $\langle U \alpha, \beta \rangle$ is defined to be a $\Z_2$-count of $J$-holomorphic currents on the symplectization $\R \times Y$ with ECH index $2$, which converge as currents to $\alpha$ (resp. to $\beta)$ when $s$ tends to $+\infty$ (resp. to $-\infty$) and pass through a fixed based point $(0,y) \in \R \times Y$, where $y \in Y$ is a generic point which is not on any (closed) Reeb orbit. Extending linearly, one obtain a map defined on the whole chain complex $ECC_*(Y,\lambda,\Gamma)$ which turns out to be a chain map, and hence, descends to a well defined map on homology. Likewise we noted for the differential, if $\langle U \alpha, \beta \rangle \neq 0$, it must hold $\mathcal{A}(\alpha) > \mathcal{A}(\beta)$, i.e., the $U$ map decreases the action.

Given a real number $L>0$, we define the filtered ECH as follows. Denote by $ECC_*^L(Y,\lambda,\Gamma)$ the $\Z_2$-vector space generated by admissible orbit sets with action $<L$. Since $\partial$ decreases the action, the latter vector space is in fact a subcomplex of $ECC_*(Y,\lambda,\Gamma)$. In this case, the \emph{$L$-filtered ECH group} is defined as the homology group of this subcomplex and is denoted by $ECH_*^L(Y,\lambda,\Gamma,J)$.

As a consequence of the existence of ECH cobordism maps, one can conclude that $[\emptyset] \neq 0 \in ECH_*(Y,\lambda,0,J)$ holds whenever $\lambda$ is \emph{symplectically fillable}, see \cite[Example 1.10]{hutchings2014lecture}. The latter means that there exists a four dimensional symplectic manifold $(X,d\tilde{\lambda})$ such that $\partial X = Y$ and $\tilde{\lambda}|_Y = \lambda$. Hutchings used this fact to define the following sequence of nontrivial quantitative invariants. Let $c_0(Y,\lambda) = 0$ and define
$$c_k(Y,\lambda) = \inf \{L \mid \exists \eta \in ECH_*^L(Y,\lambda,0,J); \ U^k \eta = [\emptyset]\},$$
for each $k \in \Z_{\geq 1}$. This is well defined as long as $\lambda$ is a nondegenerate contact form. For the degenerate case, we define 
\begin{equation}\label{cklim}
c_k(Y,\lambda) = \lim_{n\to \infty} c_k(Y,f_n\lambda),
\end{equation}
where $f_n \colon Y \to \R_{>0}$ are functions on $Y$, with $f_n{\lambda}$ nondegenerate contact forms and $\lim_{n\to \infty} f_n = 1$ in the $C^0$ topology. It follows from \cite[\S 3.1]{hutchings2011quantitative} that the limit in \eqref{cklim} exists and does not depend on the sequence $f_n$. Hence, the ECH spectrum is well defined for any contact form $\lambda$ on $Y$.

Although we need some choices to define the differential and the $U$ map, Taubes proved that ECH and the $U$ map do not depend on the almost complex structure $J$ and neither on the contact form $\lambda$.

\begin{theorem}[\cite{taubes2010embedded,taubes2010embeddedV}]\label{taubes}
There is a canonical isomorphism of relatively graded modules
$$ECH_*(Y,\lambda,\Gamma,J) = \widehat{HM}^{-*}(Y,\mathfrak{s}_\xi + PD(\Gamma)).$$
Moreover, the $U$ map defined on ECH agrees with the analogous $U$ map on the Seiberg-Witten Floer cohomology.
\end{theorem}
In the previous theorem, $\widehat{HM}^{-*}(Y,\mathfrak{s}_\xi + PD(\Gamma))$ denotes the ``from" version of Seiberg-Witten Floer cohomology defined by Kronheimer and Mrowka in \cite{kronheimer2007monopoles}.

Since ECH does not depend on $\lambda$ or $J$, from now on, we write $ECH(Y,\xi,\Gamma)$.

\subsection{ECH of $\R P^3$}
The ECH of real projective three space $\R P^3$ is well known.
\begin{theorem}\label{echrp3}
Let $\R P^3$ be the real projective three space and $\xi_0$ be its standard tight contact structure. Then its ECH is given by
\begin{align*}
ECH_*(\R P^3,\xi_0,\Gamma) = \begin{cases}
\Z_2, & \text{if} \ * \in 2\Z_{\geq 0} \\
0, & \text{otherwise}
\end{cases}
\end{align*}
for each $\Gamma \in H_1(\R P^3)\cong \Z_2$.
\end{theorem}
This result follows from the Taubes's isomorphism in Theorem \ref{taubes} together with the computation of the Seiberg-Witten Floer cohomology in \cite[\S 3.3]{kronheimer2007monopoles} and \cite[Corollary 3.4]{kronheimer2007monopoles2}, or computing the ECH visualizing this manifold as a prequantization bundle over the sphere $S^2$ and using a Morse-Bott direct limit argument as in \cite[Theorem 7.6]{nelson2020embedded} or similarly done in \cite[Proposition 4.7]{ferreira2021symplectic}. Also, one can use an irrational Katok metric as we shall explain in Section \ref{subseckatok}. Now, via the $U$ map on Seiberg-Witten Floer cohomology, we can describe the $U$ map on ECH.

\begin{theorem}\cite[Proposition 4.9]{ferreira2021symplectic} \label{Umap}
The $U$ map for $\R P^3$ with the standard tight contact structure $U \colon ECH_*(\R P^3, \xi_0, 0) \to ECH_{*-2}(\R P^3,\xi_0,0)$ is given by $U\zeta_k = \zeta_{k-1}$, where $\zeta_k$ is the generator of $ECH_{2k}(\R P_3,\xi_0,0)$, for $k\geq 1$.
\end{theorem}

\section{Elliptic Reeb orbit via ECH}
Now we use the ECH structure of $\R P^3$ discussed above to prove Theorem \ref{contractibleind2} and Theorem \ref{embeddedind2}.

\subsection{Distinguished curve via $U$ map}
The first step is to find an interesting pseudoholomorphic curve in the symplectization $\R \times \R P^3$ as stated in the following result.

\begin{prop}\label{curves}
Let $\lambda$ be a nondegenerate linearly positive contact form on $\R P^3$ defining a tight contact structure $\xi=\ker \lambda$. For a generic symplectization-admissible almost complex structure $J$ on $\R \times \R P^3$, there exists at least one of the following embedded $J$-holomorphic curves in $\R \times \R P^3$:
\begin{enumerate}[label=(\alph*)]
\item A genus one surface with only one positive end at a Reeb orbit $\gamma_a$ with $CZ(\gamma_a) = 1$. \label{a}
\item A plane asymptotic to a Reeb orbit $\gamma_b$ with $CZ(\gamma_b) = 3$. \label{b}
\item A cylinder with positive ends at $\gamma_{c_1}$ and $\gamma_{c_2}$ such that $CZ(\gamma_{c_1}) = CZ(\gamma_{c_2}) = 1$. \label{c}
\end{enumerate}
\begin{proof}
First, we note that $\R P^3$ admits a unique tight contact structure modulo isotopies, see \cite[Theorem 2.1]{MR1786111}, and hence, $\xi$ is symplectically trivial and fillable\footnote{For any Riemannian metric $g$ on $S^2$, consider a disk cotangent bundle $D_g^*S^2 = \{p \in T^*S^2 \mid \Vert p \Vert_g \leq 1\}$ equipped with the restriction of the canonical symplectic form.}. Then, let $J$ be a symplectization-admissible almost complex structure on $\R \times \R P^3$ such that $ECH_*(\R P^3,\lambda,J,\Gamma)$ and the map
$$U\colon ECH_*(\R P^3,\lambda,J,\Gamma) \to ECH_{*-2}(\R P^3,\lambda,J,\Gamma)$$
are well defined. From Theorem \ref{taubes} and Theorem \ref{echrp3}, we get
\begin{align*}
ECH_*(\R P^3,\lambda,J,\Gamma) = \begin{cases}
\Z_2, & \text{if} \ * \in 2\Z_{\geq 0} \\
0, & \text{otherwise}
\end{cases}
\end{align*}
for each $\Gamma \in H_1(\R P^3;\Z)$. Let $\Gamma = 0$. From Theorem \ref{Umap}, the $U$ map is an isomorphism in all nonzero degrees. Then, following the discussion in Section \ref{sectgrading},  we can define a $\Z$ grading by setting $ \vert \alpha \vert = I(\alpha,\emptyset)$ for every orbit set $\alpha$ in class $\Gamma= 0 \in H_1(\R P^3;\Z)$. In this grading, the empty set has degree zero and, since $\xi$ is symplectically fillable, $[\emptyset] \neq 0$ is the generator of $ECH_0(\R P^3,\xi,0)$. Since the $U$ map is an isomorphism in degree $2$, there exists a class $x$ in $ECH_2(\R P^3,\xi,0)$ such that $Ux = [\emptyset]$. Let $\sum_{k=1}^n \alpha_k$ be a representative of $x$. By the definition of the $U$ map, we obtain $\sum_{k=1}^n [U\alpha_k] = [\emptyset]$. Hence, we conclude that there must exist at least one ECH index $2$ $J$-holomorphic current in $\R \times \R P^3$ with positive ends in an admissible orbit set $\alpha_{k_0}$ and no negative ends. From \cite[Proposition 3.7]{hutchings2014lecture}, this current must be an embedded curve $C$ with $\text{ind}(C) = I(C) = 2$, where $\text{ind}(C)$ denotes the Fredholm index of the curve $C$. Now let $\gamma_1, \ldots,\gamma_q$ be the positive ends of $C$. Then
\begin{eqnarray*}
2 = \text{ind}(C) &=&  -\chi(C) + 2c_\tau(C) + \sum_{n = 1}^q CZ_\tau(\gamma_n) \\ &=& 2g(C) - 2 + q + \sum_{n = 1}^q CZ_\tau(\gamma_n) \\ &\geq & 2g(C) -2 + 2q,
\end{eqnarray*}
where $\tau$ is a global trivialization of $\xi$ and we use the hypothesis $CZ_\tau(\gamma)>0$ for all Reeb orbit $\gamma$, i.e., $\lambda$ is linearly positive. Thus, $C$ has $q \leq 2$ positive ends and we get $2$ possibilities for $q$.
\begin{enumerate}
\item \textbf{Case $q=1$:} Fredholm index equation yields
$$2 = 2g(C) -2 +1 +CZ_\tau(\gamma),$$
and so, there are two possibilities here, $g(C) = 1$ and $CZ_\tau(\gamma) = 1$, or $g(C) = 0$ and $CZ_\tau(\gamma) = 3$. These are described in \ref{a}, and \ref{b}, respectively.
\item \textbf{Case $q = 2$:} In this case, the Fredholm index equation yields
$$ 2 = 2g(C) -2 + 2 + CZ_\tau (\gamma_1)+CZ_\tau(\gamma_2),$$
and hence, since $\lambda$ is linearly positive, $C$ has genus zero and positive ends at orbits $\gamma_1$ and $\gamma_2$ such that $CZ_\tau(\gamma_1) = CZ_\tau(\gamma_2) = 1$. This is the possibility in \ref{c}.
\end{enumerate}
\end{proof}
\end{prop}

\subsection{Global surfaces of section}

We shall also use an important concept in dynamical systems theory, which goes back to Poincaré and the planar circular restricted three-body problem, namely: global surfaces of section. For a smooth flow $\varphi_t$ on a smooth $3$-dimensional manifold $Y$, a \emph{global surface of section for $\varphi_t$} is a compact embedded surface $\Sigma \subset Y$ which satisfies the following properties
\begin{enumerate}[label=(\roman*)]
\item Each component of $\partial \Sigma$ is a periodic orbit of $\varphi_t$.
\item $\varphi_t$ is tranverse to $\Sigma\backslash \partial \Sigma$.
\item For every $y\in Y\backslash \partial \Sigma$, there exists $t_+>0$ and $t_-<0$ such that $\varphi_{t_+}(y)$ and $\varphi_{t_-}(y)$ belongs to $\Sigma\backslash \partial \Sigma$.
\end{enumerate}
A \emph{disk-like global surface of section} is a global surface of section $\Sigma$ which is diffeomorphic to a two-dimensional disk. Likewise, an \emph{annulus-like global surface of section} is a global surface of section which is diffeomorphic to an annulus. The study of global surfaces of section for Reeb flows in dimension $3$ has received significant attention, see e.g. \cite{hofer1995characterisation,hofer1998dynamics,hryniewicz2016elliptic,cristofaro2019torsion,hryniewicz2022genus, hryniewicz2022global, contreras2022existence,colin2022generic}.

Among the known results, we would like to state two that are relevant to what follows. The first one is due to Birkhoff and consists of the existence of an annulus-like global surface of section for geodesic flows on positively curved Riemannian spheres. Given an embedded closed geodesic on a Riemannian $2$-sphere we recall that there are two hemispheres determined by it. The \emph{Birkhoff annulus} is the set of unit vectors based at the geodesic that points towards one of these hemispheres.

\begin{theorem}[Chapter VI in \cite{birkhoff1927dynamical}]\label{birkhoffthm}
Let $(S^2,g)$ be a Riemannian sphere such that $K>0$ everywhere. For an embedded closed geodesic $c$, the Birkhoff annulus $B_c \subset S_gS^2$ is a positive\footnote{Here, positive means that the induced orientation of the boundary agrees with the orientation along the flow.} annulus-like global surface of section for the geodesic flow in the unit tangent bundle corresponding to $g$.
\end{theorem}

The second one is due to Hofer, Wyzocki and Zehnder and ensures the existence of a disk-like global surface of section for Reeb flows on dynamically convex three spheres.

\begin{theorem}[Theorem 1.3 in \cite{hofer1998dynamics}]\label{hwzdyn}
Let $\lambda$ be a dynamically convex contact form on $S^3$. Then there exists a simple Reeb orbit with Conley--Zehnder index $3$ that bounds a disk-like global surface of section for the Reeb flow.
\end{theorem}

In this work, we shall say that $\Sigma$ is simply a global surface of section if it is for the Reeb flow. In \cite{cristofaro2019torsion}, a criterion is provided to assert if the projection $u$ in $Y$ of a $J$-holomorphic curve $\tilde{u} = (a,u) \colon C \to \R \times Y$ is a global surface of section. Denote by $h_+(C)$ the number of positive hyperbolic orbits that are ends of a curve $C$. Moreover, denote by $\mathcal{M}^J_C$ the component of the moduli space of $J$-holomorphic curves containing a fixed curve $C$.

\begin{prop}[Proposition 3.2 in \cite{cristofaro2019torsion}]\label{gss}
Let $(Y,\lambda)$ be a nondegenerate contact three-manifold, and let $J$ be a $\lambda$-compatible almost complex structure on $\R \times Y$. Let $C$ be an irreducible $J$-holomorphic curve in $\R \times Y$ such that
\begin{enumerate}[label=(\roman*)]
\item Every $C^\prime \in \mathcal{M}^J_C$ is embedded in $\R \times Y$.
\item $g(C) = h_+(C) = 0$ and $\text{ind}(C) = 2$.
\item $C$ does not have two positive ends, or two negative ends, at covers of the same simple Reeb orbit.
\item Let $\gamma$ be a simple Reeb orbit with rotation number $\theta \in \R / \Z$. If $C$ has a positive end at a $m$-fold cover of $\gamma$, then $\gcd(m,\lfloor m\theta \rfloor) = 1$. If $C$ has a negative end at a $m$-fold cover of $\gamma$, then $\gcd(m,\lceil m\theta \rceil) = 1$.
\item $\mathcal{M}^J_C/\R$ is compact. \\
Then $\pi_Y(C) \subset Y$ is a global surface of section for the Reeb flow.
\end{enumerate}
\end{prop}

Using this criterion, we obtain the following result.

\begin{prop}\label{gss2}
Under the hypotheses in Proposition \ref{curves}, suppose in addition there is no $J$-holomorphic plane in $\R \times \R P^3$ asymptotic to a Reeb orbit with Conley--Zehnder index $2$. Then,
\begin{enumerate}[label=(\alph*)]
\item The plane in Proposition \ref{curves} \ref{b} projects to a disk-like global surface of section in $\R P^3$.
\item If the ends of the cylinder in Proposition \ref{curves} \ref{c} are not at the same Reeb orbit, i.e., $\gamma_{c_1} \neq \gamma_{c_2}$, then its projection is an annulus-like global surface of section in $\R P^3$.
\end{enumerate}
\begin{proof}
It is enough to verify that these curves satisfy conditions $(i)$ to $(v)$ in Proposition \ref{gss}. In the proof of Proposition $\ref{curves}$, we saw that the curves have ends at admissible orbit sets, namely at $\alpha_{k_0}$ and $\emptyset$. Then, every curve in the component of moduli space of one of them, has ends at admissible orbit sets and ECH index $2$. So $(i)$ follows from \cite[Proposition 3.7]{hutchings2014lecture}. Since the possible ends have odd Conley--Zehnder index, they are not positive hyperbolic, and then $(ii)$ is readily verified. For the plane, $(iii)$ is clear. For the cylinder, it follows from our assumption. Now we recall that the Conley--Zehnder indices possibilites of the ends are $1$ and $3$. In any case, $m\theta \leq 3/2$, and hence, $(iv)$ also holds. Finally, let $C$ be the plane or the cylinder in discussion and suppose that $\mathcal{M}^J_C/\R$ is not compact. It follows from \cite[Lemma 5.12]{hutchings2014lecture} that there exists a sequence of $J$-holomorphic curves in $\mathcal{M}^J_C/\R$ converging to a \emph{broken holomorphic current} $(\mathcal{C}_+,\mathcal{C}_-)$ with $I(\mathcal{C}_+)=I(\mathcal{C}_-) = 1$. Moreover, $\mathcal{C}_+$ is positive asymptotic to $\alpha_{k_0}$ and $\mathcal{C}_-$ has no negative ends. By \cite[Proposition 3.7]{hutchings2014lecture}, $C_- := \mathcal{C}_-$ must be an embedded curve and $I(C_-) = \text{ind}(C_-)=1$. Let $\gamma^+_1,\ldots, \gamma^+_q$ be the positive ends of $C_-$. The Fredholm index equation yields
\begin{eqnarray*}
1 = \text{ind}(C_-) &=&  -\chi(C_-) + 2c_\tau(C_-) + \sum_{n = 1}^q CZ_\tau(\gamma^+_n) \\ &=& 2g(C_-) - 2 + q + \sum_{n = 1}^q CZ_\tau(\gamma^+_n) \\ &\geq & 2g(C_-) -2 + 2q,
\end{eqnarray*}
and so, $q=1$ and $g(C_-) = 0$. Therefore, $C_-$ must be a plane asymptotic to a Reeb orbit with Conley--Zehnder index $2$. This contradicts our hypotheses and then, $(v)$ must hold.
\end{proof}
\end{prop}

\begin{rmk}
One can repeat the discussion above for the case of $Y = S^3$ and combining Propositions \ref{curves} and \ref{gss2}, it recovers Theorem \ref{hwzdyn} in the nondegenerate case.
\end{rmk}

\subsection{Elliptic Reeb orbits on some $\R P^3$}

Propositions \ref{curves} and \ref{gss2} together with a characterization of the tight $S^3$ in \cite{hofer1995characterisation}, lead us to the proof of Theorem \ref{contractibleind2}.

\begin{proof}[Proof of Theorem \ref{contractibleind2}]
Since $\pi_1(\R P^3) = \Z_2$, $\lambda$ does not admit a hyperbolic Reeb orbit with Conley--Zehnder index $1$. Indeed, if it does, the double iterate of such an orbit would be a contractible orbit with Conley--Zehnder index $2$ which contradicts our hypothesis. Now, we claim that the three cases in Proposition \ref{curves} imply the existence of an elliptic orbit with Conley--Zehnder index $1$. It is clear for cases \ref{a} and \ref{c}. Suppose then \ref{b} holds, i.e., there exists an ECH index $2$ $J$-holomorphic plane in $\R \times \R P^3$ asymptotic to a Reeb orbit $\gamma$ such that $CZ(\gamma) = 3$. If $\gamma$ were simple, it would be a contractible, simple and nondegenerate Reeb orbit. Moreover, by Proposition \ref{gss2}, $\gamma$ would bound a disk-like global surface of section for the Reeb flow in $\R P^3$, namely the projection of the latter plane. In this situation, \cite[Theorem 1.4]{hofer1995characterisation} leads us to a contradiction yielding that $(\R P^3,\xi)$ is contactomorphic to the tight $(S^3,\xi_0)$. Therefore, $\gamma$ cannot be simple and must be a covering of an elliptic Reeb orbit with Conley--Zehnder index $1$.
\end{proof}

Similarly, we prove Theorem \ref{embeddedind2}.

\begin{proof}[Proof of Theorem \ref{embeddedind2}]
With the additional hypothesis of $c_1(\R P^3,\lambda) = \mathcal{A}_{min}^0(\lambda)$, we can take the orbit set $\alpha_{k_0}$ such that $\mathcal{A}(\alpha_{k_0}) = \mathcal{A}_{min}^0(\lambda)$. In particular, the moduli space of planes as in \ref{b} in Proposition \ref{curves} is still compact. Moreover, by the hypothesis on the Conley--Zehnder index $2$ Reeb orbits, the possibilities \ref{a} and \ref{c} in Proposition \ref{curves} still fit in the conclusion. Hence the proof follows in the same way of the proof of Theorem \ref{contractibleind2}. The action range claimed in the statement of Theorem \ref{embeddedind2} follows from the fact that $\mathcal{A}(\alpha_{k_0}) = \mathcal{A}_{min}^0(\lambda)$ is an element in the set $\{\mathcal{A}(\gamma_a),\mathcal{A}(\gamma_b),\mathcal{A}(\gamma_{c_1})+\mathcal{A}(\gamma_{c_2})\}$.
\end{proof}

\section{ECH of Finsler Spheres}\label{finsler}
\subsection{From Finsler to Contact Geometry}
In this section, we follow \cite{hryniewicz2013introduccao} and summarize the dictionary relating Finsler metrics on manifolds and the Hilbert contact form on the unit tangent bundle associated to this metric.

Let $N$ be a smooth manifold and $F \colon TN \to [0,+\infty)$ be a Finsler metric on $N$, that is, $F$ is a continuous map which satisfies
\begin{enumerate}[label=(\roman*)]
\item (smoothness) $F$ is smooth on $TN\backslash N$, i.e., away from the zero section.
\item (homogeneity) $F(tv) = tF(v)$, for all $v \in TN$ and $t \in \R_{>0}$.
\item (convexity) The symmetric bilinear form 
\begin{eqnarray}\label{bilinearfinsler}
g_v\colon T_qN \times T_qN &\to& \R \\ \nonumber
(u,w) &\mapsto& \frac{1}{2} \frac{\partial^2}{\partial t \partial s} F^2(v+su+tw)|_{s=t=0}
\end{eqnarray}
is positive-definite for all $v \in T_qN\backslash \{0\}$ and every $q \in N$. 
\end{enumerate}

Note that given a Riemannian metric $g$ on the manifold $N$, one can define a Finsler metric by $F(v) := \sqrt{g(v,v)}$, and hence, Finsler metrics generalizes the notion of Riemannian metrics. Moreover, a Finsler metric gives a natural identification between the tangent bundle and the cotangent bundle of the given manifold $N$. In fact, the \emph{Legendre transformation} defined by
\begin{eqnarray}\label{legendretr}
\mathcal{L}_F \colon TN\backslash N &\to& T^*N\backslash N \\ \nonumber
v &\mapsto& g_v(v,\cdot)
\end{eqnarray}
is a diffeomorphism.

We then define the unit tangent bundle of $N$ and the unit cotangent bundle of $N$ by
$$S_FN := F^{-1}(1) = \{v \in TN \mid F(v) = 1\},$$
and
$$S_F^*N := \{p \in T^*N \mid F(\mathcal{L}_F^{-1}(p)) = 1\},$$
respectively. Note that these two are odd dimensional manifolds (codimension $1$ submanifolds on the bundles $TN$ and $T^*N$, respectively). It is well known that the tautological one form $\lambda_{taut}$ on $T^*N$ restricts to a contact form on $S_F^*N$. Moreover, similarly to the definition of $\lambda_{taut}$, one defines the Hilbert form $\lambda_F$ on $TN$ by
\begin{equation}\label{hilbform}
(\lambda_F)_v(\zeta) = g_v(v,d\pi \cdot \zeta),
\end{equation}
for any $v \in TN$, $\zeta \in T_vTN$ and where $\pi \colon TN \to N$ is the natural projection. In fact, the Legendre transformation $\mathcal{L}_F$ interchanges these two forms, i.e., $\mathcal{L}_F^*\lambda_{taut} = \lambda_F$. In particular, $\lambda_F$ restricts to a contact form on the unit tangent bundle $S_FN$.

It is a simple exercise to check that the contact structure $\xi_F = \ker \lambda_F$ does not depend essentially on the metric. More precisely, given two Finsler metrics $F_1,F_2$ on $N$, there exists a contactomorphism $\phi\colon (S_{F_1}N,\xi_{F_1}) \to (S_{F_2}N,\xi_{F_2})$, meaning that $\phi$ is a diffeomorphism such that $\phi_* \xi_{F_1} = \xi_{F_2}$. In particular, for $N = S^2$, the round metric $g_0$ is such that $(S_{g_0}S^2,\xi_{g_0})$ is contactomorphic to the standard tight $(\R P^3,\xi_0)$, and hence, the contact structure defined by any Finsler metric on the sphere $S^2$ is tight.

The following result shows that the Reeb vector field for the Hilbert form $\lambda_F$ agrees with a well known vector field in differential geometry.

\begin{prop}[Teorema 4.4.10 in \cite{hryniewicz2013introduccao}] \label{flowsunit}
Let $F$ be a Finsler metric defined on a smooth manifold $N$. Then the Reeb vector field of $(S_FN,\lambda_F)$ agrees with the Geodesic vector field for $F$.
\end{prop}

Thus, given a closed geodesic parametrized by arc length $c \colon I \to N$, one has a corresponding Reeb orbit $\gamma = (c,\dot{c})\colon I \to S_FN$. In fact, whenever $F$ is reversible\footnote{That is, $F(v) = F(-v) \ \forall v \in TN$.}, such a geodesic $c$ gives rise to two Reeb orbits $(c,\pm \dot{c})$ on $(S_FN,\lambda_F)$. It is simple to check that
$$\mathcal{A}(c,\dot{c}) = \int_{(c,\dot{c})} \lambda_F = \int_I F(\dot{c}) dt = Length(c).$$
Further, the linearized Poincaré map for the Reeb orbit $\gamma = (c,\dot{c})$ is conjugated to the linear Poincaré map for the geodesic $c$ defined using Jacobi fields, see  \cite[Lemma 2.3]{hryniewicz2013global}. In addition, it follows from \cite{liu2005relation} that the Conley--Zehnder index of $\gamma$ with respect to a trivialization of $\xi_F$ that extends to a disk bounding the orbit coincides with the Morse index of the closed geodesic $c$.

\subsection{First value on ECH spectrum for $1/4$-pinched metrics}
In this section we compute the first value on the ECH spectrum for the $1/4$-pinched Riemannian case, proving Theorem \ref{thmc1}. We shall use two lemmas. The first one uses Klingenberg \cite[Theorem 2.6.9]{klingenberg1995riemannian} and Toponogov \cite[Theorem 2.7.12]{klingenberg1995riemannian} estimates to guarantee that $\mathcal{A}_{min}^0(\lambda_g) = 2L$ for the $1/4$-pinched Riemannian case.

\begin{lemma}\label{a0=2Llemma}
Let $(S^2,g)$ be a Riemannian sphere such that $1/4 < K \leq 1$, where $K$ is the sectional curvature. Hence
\begin{equation}\label{a0=2L}
4\pi \leq \mathcal{A}_{min}^0(\lambda_g) = 2L < 8\pi,
\end{equation}
where $L$ is the length of a shortest geodesic for $g$.
\begin{proof}
Note that a null-homologous orbit set with minimal action must be of the form $\gamma$ such that $\mathcal{A}(\gamma) = \mathcal{A}^0_{min}(\lambda_g)$ ,$\gamma_1\gamma_2$ or $\widetilde{\gamma}^2$, where each of these latter Reeb orbits corresponds to closed geodesics with minimal length on $(S^2,g)$. This holds due to $\pi_1(\R P^3) = \Z_2$ and to the fact that a smooth curve $c \colon I \to S^2$ parametrized by arc length and which all self-intersections are transverse induces a contractible curve $(c,\dot{c}) \colon I \to S_gS^2 \cong \R P^3$ if, and only if, $c$ has an odd number of self-intersections. Hence, $\mathcal{A}_{min}^0(\lambda_g)$ must agree with the smallest element in the set $\{\mathcal{A}(\gamma),\mathcal{A}(\gamma_1)+\mathcal{A}(\gamma_2),2\mathcal{A}(\widetilde{\gamma})\} = \{\mathcal{A}(\gamma),2L\}$, where $L$ is the length of a shortest closed geodesic. Since $1/4 < K \leq 1$, it follows from Klingenberg and Toponogov comparison theorems that a closed geodesic for $g$ either is simple with length in the interval $[2\pi, 4\pi)$ or have at least two self-intersections and length $\geq 6\pi$. Therefore, if $\gamma$ is a null-homologous Reeb orbit, $\gamma$ must correspond to a closed geodesic $c_\gamma$ with at least three self-intersections and by Klingenberg's estimate\footnote{Klingenberg's estimate ensures that the injectivity radius, $\text{inj}(p)$, is at least $\pi$ for every $p$ in $S^2$, yielding that geodesic loops have length $\geq 2\pi$.}, $\mathcal{A}(\gamma) = Length(c_\gamma) \geq 8 \pi$, and hence, we obtain \eqref{a0=2L}.
\end{proof}
\end{lemma}

The second Lemma will be useful to extend the computation of $c_1(S_gS^2,\lambda_g)$ to the degenerate case.

\begin{lemma}\label{substoL}
Let $\{g_n\}_{n \in \mathbb{N}}$ be a sequence of Riemannian metrics with sectional curvature $K_n>0$ on the two dimensional sphere $S^2$ converging to a positively curved Riemannian metric $g$ in the $C^0$-topology. Suppose that the sequence consisting of lengths, $L_n$, of shortest closed geodesics for $g_n$ converges to a positive real number $L$. Then $L$ is the length of a shortest closed geodesic for $g$.
\begin{proof}
By a well known Calabi-Cao result \cite[Theorem D]{calabi1992simple} (see also the Appendix due to Abbondandolo and Mazzucchelli in \cite{benedetti2022relative}), $L_n$ is the length of a simple closed geodesic $\gamma_n$ which is the Birkhoff minmax geodesic:
$$L_n = \mathrm{bir}(S^2,g_n) := \inf_{u \in \mathcal{U}} \max_{z\in [-1,1]} E_n(u(z))^{1/2}.$$
Here, for each $n\in \mathbb{N}$, $E_n$ denotes the energy functional
\begin{eqnarray*}
E_n \colon W^{1,2}(S^1,S^2) &\to& [0,\infty) \\
\zeta &\mapsto& E_n(\zeta) = \int_{S^1} \Vert \dot{\zeta} \Vert_{g_n} dt
\end{eqnarray*}
defined on the $W^{1,2}$ free loop space and $\mathcal{U} \subset C^0([-1,1],W^{1,2}(S^1,S^2))$ is the space of \emph{sweepouts} consisting of suitable one parameter families of closed curves starting and ending at point curves. Similarly, since $g$ is also positively curved, the length of a shortest geodesic for $g$ coincides with $\mathrm{bir}(S^2,g)$. The Lemma then follows from the continuity of the Birkhoff minmax value with respect to the metric. In fact, one can check that
$$L = \lim_{n \to +\infty} L_n = \lim_{n \to +\infty} \mathrm{bir}(S^2,g_n) = \mathrm{bir}(S^2,g).$$
\end{proof}
\end{lemma}

In particular, the conclusion in the Lemma holds if $\{g_n\}_{n \in \mathbb{N}}$ is a sequence satisfying $K_n \geq \delta>0$ and $C^2$-converges to a Riemannian metric $g$, since in this case $K = \lim_{n \to +\infty} K_n$ is automatically positive. Now we are ready to prove Theorem \ref{thmc1}.

\begin{proof}[Proof of Theorem \ref{thmc1}]
Suppose first that $g$ is a bumpy metric. In this case, the contact form $\lambda_g$ is nondegenerate, and hence, for a generic almost complex structure $J$ we have a well defined homology $ECH_*(S_gS^2,\lambda_g,\Gamma,J)$. Since $1/4 < K \leq 1$, it follows from \cite[Theorem 4.2]{MR727711} that any shortest closed geodesic\footnote{That is, a closed geodesic with minimal length among all closed geodesics for $g$.} for $g$ is simple and has Morse index $1$. Let $\gamma$ and $\overline{\gamma}$ be the two Reeb orbits corresponding to a shortest closed geodesic traversed in both directions on $S^2$. Note that we have
$$\mathcal{A}(\gamma\overline{\gamma}) = 2L = \mathcal{A}_{min}^0(\lambda_g),$$
where the last equality follows from Lemma \ref{a0=2Llemma}.
Since $c_1(S_gS^2,\lambda_g) \geq \mathcal{A}_{min}^0(\lambda_g)$, it is enough to prove that $\gamma\overline{\gamma}$ represents an element in homology and $U([\gamma\overline{\gamma}]) = [\emptyset]$.

First, we claim that $\gamma \overline{\gamma}$ is closed, that is, $\partial (\gamma \overline{\gamma}) = 0$. Note that $\langle \partial (\gamma\overline{\gamma}),\beta \rangle = 0$ for $\beta \neq \emptyset$ because the differential decreases the action and $\gamma \overline{\gamma}$ has the minimal action among the null-homologous orbit sets. Moreover, if $\langle \partial(\gamma \overline{\gamma}),\emptyset \rangle \neq 0$, there would exist an embedded $J$-holomorphic curve $C$ in $\R \times S_gS^2$ such that $I(C)=\text{ind}(C) = 1$. In this case,
\begin{eqnarray*}
1 = \text{ind}(C) &=& 2g(C) -2 + 2 + CZ(\gamma) + CZ(\overline{\gamma}) \\ &=& 2g(C) + 1 + 1 \geq 2
\end{eqnarray*}
leads us to a contradiction. Here we used that $CZ(\gamma) = CZ(\overline{\gamma}) = 1$ agrees with the Morse index of the corresponding geodesic which is equal to $1$. This proves the claim.

To prove that $U([\gamma\overline{\gamma}]) = [\emptyset]$, let $(0,y) \in \R \times S_gS^2$ be a fixed base point such that $y \in S_gS^2$ is a generic point which is not on any (closed) Reeb orbit and such that the $U$ map $U_{y,J}\colon ECC_*(S_gS^2,\lambda_g,0,J) \to ECC_*(S_gS^2,\lambda_g,0,J)$ is well defined. Since the $U$ map decreases the action, we have $\langle U(\gamma \overline{\gamma}), \beta \rangle = 0$ for every $\beta \neq \emptyset$ again.

By Theorem \ref{birkhoffthm}, the Birkhoff annulus $B_\gamma \subset S_gS^2$ is a global surface of section with boundary $\partial B_\gamma = \gamma \cup \overline{\gamma}$. This yields the existence of an open book decomposition of $S_gS^2$ supporting $\ker \lambda_g$, where $\gamma \cup \overline{\gamma}$ is the binding and whose pages are diffeomorphic to annuli. It follows from \cite[Proposition 3.16]{hryniewicz2022genus} that there exists (possibly other) open book decomposition supporting $\ker \lambda_g$ whose pages are annulus-like global surfaces of section for the Reeb flow on $(S_gS^2,\lambda_g$) and are projections of curves in $\mathcal{M}^J(\gamma\overline{\gamma},\emptyset)$. The latter denotes the moduli space of embedded genus zero $J$-holomorphic curves in $\R \times S_gS^2$ with exactly two positive ends converging asymptotically to $\gamma$ and $\overline{\gamma}$, and no negative ends.

Note that given a curve $C \in \mathcal{M}^J(\gamma\overline{\gamma},\emptyset)$, we have
$$\text{ind}(C) = 2g(C) - 2 + 2 +CZ(\gamma) + CZ(\overline{\gamma}) = 1 + 1 = 2.$$
Since $C$ is embedded, $I(C) = \text{ind}(C)$ must hold by the \emph{Index inequality} in \cite[p. 41]{hutchings2014lecture}. We can take $C$ as being an element in $\mathcal{M}^J(\gamma\overline{\gamma},\emptyset)$ such that the projection $\pi_{S_gS^2}(C)$ is the unique page whose $y$ lies in the interior and, by translating in the $\R$ component, we can suppose that $(0,y) \in C$. Hence, $C$ is a curve counted in the coefficient $\langle U_{y,J} (\gamma \overline{\gamma}), \emptyset \rangle$.

We claim that there is no other curve counted in this coefficient. Indeed, let $C^\prime$ be a $J$-holomorphic curve counted in $\langle U_{y,J} (\gamma \overline{\gamma}), \emptyset \rangle$. By the definition of the $U$ map, $C^\prime$ is an element in $\mathcal{M}^J(\gamma\overline{\gamma},\emptyset)$ such that $I(C^\prime) = 2$ and $(0,y) \in C^\prime$. Since $1/4 < K \leq 1$, the contact form $\lambda_g$ is dynamically convex and then satisfies the hypotheses in Propositions \ref{curves} and \ref{gss2}. Thus, Proposition \ref{gss} guarantees that the projection $\pi_{S_gS^2}(C^\prime)$ is an annulus-like global surface of section. In this case, we have $y \in \pi_{S_gS^2}(C) \cap \pi_{S_gS^2}(C^\prime)$, and hence, $C^\prime$ must be equal to $C$. This equality holds because \cite[Proposition 3.5]{hryniewicz2022genus} ensures that the projection of two curves in $\mathcal{M}^J(\gamma\overline{\gamma},\emptyset)$ are either equal or disjoint. Therefore, $C$ is the unique curve counted in the coefficient $\langle U_{y,J} (\gamma \overline{\gamma}), \emptyset \rangle$, yielding $U_{y,J}(\gamma \overline{\gamma}) = \emptyset$. In particular, it follows from Proposition \ref{Umap} that $\gamma \overline{\gamma}$ represents a nonzero class in homology (the generator $\zeta_2$ of $ECH_2(S_gS^2,\lambda_g,0,J)$).

By the definition of $c_1(S_gS^2,\lambda_g)$ in the nondegenerate case, we conclude $$c_1(S_gS^2,\lambda_g) \leq \mathcal{A}(\gamma \overline{\gamma}) = \mathcal{A}^0_{min}(\lambda_g),$$ and this finishes the proof for the bumpy case.

For the general case, let $\lambda$ be the degenerate contact form on $\R P^3$ corresponding to $\lambda_g$ on $S_gS^2$, that is, $(\R P^3,\lambda)$ is strictly contactomorphic to $(S_gS^2,\lambda_g)$. Recall that by definition in the degenerate case in \eqref{cklim}, we have
$$c_1(S_gS^2,\lambda_g) = c_1(\R P^3,\lambda) = \lim_{n \to \infty} c_1(\R P^3,f_n\lambda),$$
where $f_n$ is any sequence of positive functions such that $f_n\lambda$ are nondegerate contact forms and $\lim_{n \to \infty} f_n = 1$ in the $C^0$ topology. Note that by the bumpy metrics theorem \cite[Theorem 1]{anosov1983generic}, there must exist a sequence of $1/4$-pinched bumpy metrics $\{g_n\}$ converging to $g$ in the $C^\infty$ topology. In particular, if $f_{g_n} \colon\R P^3 \to \R_{>0}$ is the function such that $(\R P^3,f_{g_n}\lambda) \cong (S_{g_n}S^2,\lambda_{g_n})$ for each $n$, we have $\lim_{n \to \infty} f_{g_n} = 1$ in the $C^0$ topology provided the convergence $g_n \to g$. Hence, it is enough to compute the limit
$$\lim_{n \to \infty} c_1(\R P^3,f_n\lambda) = \lim_{n \to \infty} c_1(S_{g_n}S^2,\lambda_{g_n}) = \lim_{n \to \infty} 2 L_n,$$
where the last equality follows from the proof above for the bumpy case and $L_n$ denotes the length of a shortest geodesic for $g_n$.

By the pinching condition, Lemma \ref{a0=2Llemma} confirms that $2\pi \leq L_n < 4\pi$, for each $n$, and thus, there exists a subsequence of $\{L_n\}_{n \in \mathbb{N}}$ converging to $L \geq 2\pi >0$. Lemma \ref{substoL} ensures that $L$ is the length of a shortest geodesic for $g$. Putting all these together, we conclude that
$$ c_1(S_gS^2,\lambda_g) = \lim_{n \to \infty} c_1(S_{g_n}S^2,\lambda_{g_n}) = \lim_{n \to \infty} 2 L_n = 2L = \mathcal{A}_{min}^0(\lambda_g),$$
using again Lemma \ref{a0=2Llemma} to obtain the last equality.
\end{proof}

\subsection{ECH of irrational Katok example}\label{subseckatok}
Now we study the ECH of irrational Katok metrics and compute its ECH spectrum. First, we follow \cite{ziller1983geometry} and summarize Katok's example. Let $g_0$ be the round metric on $S^2 \subset \R ^3$, $a \in \R$, and consider the Hamiltonian $H_a \colon T^*S^2 \to \R$ defined by
$$H_a(p) = \Vert p \Vert_{g_0}^* + a p(\partial_\theta),$$
where $\Vert p \Vert_{g_0}^*$ is the dual norm (with respect to the norm induced by $g_0$), and $\partial_\theta$ is the Killing vector field generating the rotations around $z$-axis on $S^2 \subset \R^3$. Namely, $\partial_\theta = (-y \partial_x + x \partial_y)$ in cartesian coordinates $(x,y,z) \in \R ^3$. Consider the Legendre transformation associated to $\frac{1}{2}H_a^2$:
\begin{eqnarray*}
\mathcal{L}_{\frac{1}{2}H_a^2} \colon T^*S^2 &\to& TS^2 \\ p &\mapsto& H_a(p)\left(\frac{(g^b_0)^{-1}(p)}{\Vert p \Vert^*}+a\partial_\theta\right).
\end{eqnarray*}
Here $g_0^b$ denotes the usual bundle isomorphism
\begin{eqnarray*}
g_0^b\colon TS^2 &\to& T^*S^2 \\ v &\mapsto& g_0(v,\cdot)
\end{eqnarray*}
induced by the round metric $g_0$. Moreover, it is straightforward to check that 
$$F_a := H_a \circ \mathcal{L}_{\frac{1}{2}H_a^2}$$
is a Finsler metric when $\vert \alpha \vert <1$. This metric can be interpreted as the metric obtained by perturbing the round metric $g_0$ on $S^2$ in the direction of the \emph{wind} $\partial_\theta$, meaning that distances are now computed considering a small contribution in favor of a rotation around the $z$-axis in $\R^3$. 

One can check that for $a \in \R\backslash \mathbb{Q} \cap (0,1)$, the closed geodesics for $F_a$ are exactly the closed geodesics for the round metric $g_0$ which are invariant by the rotations around $z$-axis. Hence, after the wind $\partial_\theta$ perturbation, the only closed geodesics of $g_0$ that survive are the equators. In this case, $F_a$ has exactly two closed geodesics $c_1,c_2$ with lengths $2\pi/(1+a)$ and $2\pi/(1-a)$, corresponding to the equator traversed in or opposite to the direction of the rotation, respectively. Moreover, the linear Poincaré maps $P_{c_1}, P_{c_2}$ corresponding to these geodesics are conjugated to rotations with angle $2\pi/(1+a)$ and $2\pi/(1-a)$, respectively. From now on, we fix $a \in \R\backslash \mathbb{Q} \cap (0,1)$.

\subsubsection{The chain complex}
Translating the latter facts to the contact topology side, we get  
the $3$-dimensional closed contact manifold $(S_{F_a}S^2, \lambda_{F_a})$ admitting exactly two elliptic Reeb orbits $\gamma_1, \gamma_2$ with actions $2\pi/(1+a)$ and $2\pi/(1-a)$, respectively. Further, there is a symplectic global trivialization $\tau$ of the contact structure $\xi_{F_a} := \ker \lambda_{F_a}$ on $S_{F_a}S^2$ such that
$$CZ_\tau(\gamma_1^k) = 2 \left \lfloor \frac{k}{1+a} \right \rfloor +1 \quad \text{and} \quad CZ_\tau(\gamma_2) = 2 \left \lfloor \frac{k}{1-a} \right \rfloor +1.$$
Since $\lambda_{F_a}$ is nondegenerate, the ECH chain complex $ECC_*(S_{F_a}S^2,\lambda_{F_a},\Gamma,J)$ is well defined for a generic symplectization-admissible almost complex structure on $\R \times S_{F_a}S^2$. Moreover, since the Reeb orbits $\gamma_1$ and $\gamma_2$ are elliptic, the index parity property of ECH index \cite[Proposition 1.6]{hutchings2002index} yields that the ECH index between two generators is always an even number. Thus, the differential
$$\partial \colon ECC_*(S_{F_a}S^2,\lambda_{F_a},\Gamma,J) \to ECC_*(S_{F_a}S^2,\lambda_{F_a},\Gamma,J)$$
vanishes for any $J$ and, therefore, the homology $ECH_*(S_{F_a}S^2,\lambda_{F_a},\Gamma)$ agrees with its chain complex and is generated by orbit sets $\alpha = \{(\gamma_1,m_1),(\gamma_2,m_2)\}$ such that
$$[\alpha] = m_1 [\gamma_1] + m_2 [\gamma_2] = \Gamma \in H_1(S_{F_a}S^2;\Z) \cong \Z_2.$$
This last condition is equivalent to $m_1 +m_2 \equiv \Gamma \mod 2$ identifying $\Gamma \in \Z_2$ since the projections of the Reeb orbits are simple closed geodesics on the sphere $S^2$, and so, they cannot be null-homologous on $S_{F_a}S^2$.

\subsubsection{Grading by ECH Index}
We now define an absolute $\Z$ grading on $ECH_*(S_{F_a}S^2,\lambda_{F_a},\Gamma)$. Recall that $\xi_{F_a}$ is a trivial contact structure and so $c_1(\xi) + 2 PD(\Gamma) = 0$ for each $\Gamma$, where $PD(\Gamma)$ denotes the Poincaré dual of $\Gamma$. In this case,
$$\vert \alpha \vert := I(\alpha, \emptyset) \quad \text{and} \quad \vert \beta \vert := I(\beta, \gamma_1)$$
define absolute $\Z$ gradings on $ECH_*(S_{F_a}S^2,\lambda_{F_a},\Gamma)$ for $\Gamma = \overline{0}$ and $\Gamma = \overline{1}$, respectively.

\begin{lemma}\label{gradingkatok}
The gradings defined above are given by
\begin{equation}
\vert \gamma_1^{m_1} \gamma_2^{m_2} \vert = 2 \left( \frac{m_1+m_2}{2} - \frac{m_1^2}{4} + \frac{m_1 m_2}{2} - \frac{m_2^2}{4} + \sum_{k=1}^{m_1} \left \lfloor \frac{k}{1+\alpha} \right \rfloor + \sum_{k=1}^{m_2} \left \lfloor \frac{k}{1-\alpha} \right \rfloor \right), \label{gradingnull}
\end{equation}
when $m_1+m_2 \equiv 0 \mod 2$, and
\begin{equation}
\vert \gamma_1^{n_1} \gamma_2^{n_2} \vert = 2\left(\frac{(n_1+n_2)}{2} -\frac{n_1^2}{4} + \frac{n_1n_2}{2} -\frac{n_2^2}{4} -\frac{1}{2} + \sum_{k=1}^{n_1} \left \lfloor \frac{k}{1-a} \right \rfloor + \sum_{k=1}^{n_2}\left \lfloor \frac{k}{1-a} \right \rfloor\right),
\end{equation}
when $n_1+n_2 \equiv 1 \mod 2$.
\begin{proof}
Let $\alpha = \gamma_1^{m_1} \gamma_2^{m_2} \in \Gamma = 0$, i.e., $m_1+m_2$ is even. Since $H_2(S_{F_a}S^2;\Z) \cong H_2(\R P^3;\Z) = 0$, there exists a unique class $Z \in H_2(S_{F_a}S^2,\alpha,\emptyset)$. We shall use the global trivialization $\tau$ of $\xi_{F_a}$ mentioned above, and hence, the term $c_\tau$ vanishes identically. Similarly to the proof of \cite[Proposition 4.4]{ferreira2021symplectic}, we compute
\begin{eqnarray*}
Q_\tau(Z) &=& \frac{1}{4}Q_\tau(2Z) = \frac{1}{4}(m_1^2 sl(\gamma_1^2) + 2m_1m_2 lk(\gamma_1,\gamma_2) + m_2^2 sl(\gamma_2^2)) \\ &=& \frac{1}{4}(-2m_1^2 + 4m_1m_2 - 2m_2^2) = \frac{-m_1^2}{2}+m_1m_2 + \frac{-m_2^2}{2}.
\end{eqnarray*}
Therefore,
\begin{eqnarray*}
\vert \gamma_1^{m_1} \gamma_2^{m_2} \vert &=& Q_\tau(Z) + \sum_{k=1}^{m_1}CZ_\tau(\gamma_1^k) + \sum_{k=1}^{m_2} CZ_\tau(\gamma_2^k) \\ &=& \frac{-m_1^2}{2}+m_1m_2 + \frac{-m_2^2}{2} + \sum_{k=1}^{m_1} 2 \left \lfloor \frac{k}{1+a} \right \rfloor +1 + \sum_{k=1}^{m_2} 2 \left \lfloor \frac{k}{1-a} \right \rfloor +1 \\ &=& \frac{-m_1^2}{2}+m_1m_2 + \frac{-m_2^2}{2} + m_1 + m_2 + \sum_{k=1}^{m_1} 2 \left \lfloor \frac{k}{1+a} \right \rfloor + \sum_{k=1}^{m_2} 2 \left \lfloor \frac{k}{1-a} \right \rfloor \\ &=& 2 \left( \frac{m_1+m_2}{2} - \frac{m_1^2}{4} + \frac{m_1 m_2}{2} - \frac{m_2^2}{4} + \sum_{k=1}^{m_1} \left \lfloor \frac{k}{1+\alpha} \right \rfloor + \sum_{k=1}^{m_2} \left \lfloor \frac{k}{1-\alpha} \right \rfloor \right).
\end{eqnarray*}
For $\beta = \gamma_1^{n_1}\gamma_2^{n_2} \in \Gamma = \overline{1}$, let $W$ be the unique class in $H_2(S_{F_a}S^2,\beta,\gamma_1)$. If $S_0$ is a representative of the class in $H_2(S_{F_a}S^2,\gamma_1^{2n_1-2}\gamma_2^{2n_2},\emptyset)$, we can take $S_0+\R \times \gamma_1^2$ as a representative of $2W$. Then we compute
\begin{eqnarray*}
Q_\tau(W) &=& \frac{1}{4}Q_\tau(2W) = \frac{1}{4}Q_\tau(S_0+\R \times \gamma_1^2) \\ &=& \frac{1}{4}\Big(Q_\tau(S_0) + 2Q_\tau(S_0,\R \times \gamma_1^2)\Big) \\ &=& \frac{1}{4}\Big(-2(n_1-1)^2 + 4(n_1-1)n_2 + -2 n_2^2 + -4(n_1-1) + 4n_2\Big) \\ &=& \frac{-n_1^2}{2} +n_1n_2 + -\frac{n_2^2}{2}
\end{eqnarray*}
since $Q_\tau(S_0) = (n_1-1)^2sl(\gamma_1^2) + 2(n_1-1)n_2lk(\gamma_1^2,\gamma_2^2) + n_2^2sl(\gamma_1^2)$ and
$$Q_\tau(S_0,\R \times \gamma_1^2) = (n_1-1)sl(\gamma_1^2) + n_2lk(\gamma_1^2,\gamma_2^2).$$
Thus the grading is given by
\begin{eqnarray*}
\vert \gamma_1^{n_1}\gamma_2^{n_2} \vert &=& Q_\tau(W) + \sum_{k=1}^{n_1}CZ_\tau(\gamma_1^k) + \sum_{k=1}^{n_2} CZ_\tau(\gamma_2^k) - CZ_\tau(\gamma_1) \\ &=& -\frac{n_1^2}{2} +n_1n_2  -\frac{n_2^2}{2} + \sum_{k=1}^{n_1}\left(2 \left \lfloor \frac{k}{1+a} \right \rfloor +1\right) + \sum_{k=1}^{n_2}\left(2 \left \lfloor \frac{k}{1-a} \right \rfloor +1\right) - 1\\ &=& -\frac{n_1^2}{2} +n_1n_2 -\frac{n_2^2}{2}  +n_1  +n_2 + \sum_{k=1}^{n_1}2 \left \lfloor \frac{k}{1+a} \right \rfloor + \sum_{k=1}^{n_2}2 \left \lfloor \frac{k}{1-a} \right \rfloor -1 \\ &=& 2\left(\frac{(n_1+n_2)}{2} -\frac{n_1^2}{4} + \frac{n_1n_2}{2} -\frac{n_2^2}{4} -\frac{1}{2} + \sum_{k=1}^{n_1} \left \lfloor \frac{k}{1-a} \right \rfloor + \sum_{k=1}^{n_2}\left \lfloor \frac{k}{1-a} \right \rfloor\right).
\end{eqnarray*}
\end{proof}
\end{lemma}

Now we are ready to compute the ECH groups for the irrational Katok metric on the sphere. 

\begin{prop}\label{echkatok}
The ECH chain complex of the unit tangent bundle of the sphere for a irrational Katok metric $F_a$ is given by
\begin{align*}
ECC_*(S_{F_a}S^2,\lambda_{F_a},\Gamma) = \begin{cases}
\Z_2, & \text{if} \ * \in 2\Z_{\geq 0} \\
0, & \text{else}
\end{cases}
\end{align*}
for each $\Gamma \in H_1(S_{F_a}S^2;\Z) \cong \Z_2$.
\begin{proof}
It is readily verified that both gradings in Lemma \ref{gradingkatok} are always positive even integers. We claim that these give bijections between the generators and $2\Z_{\geq 0}$ for the two homology classes in $H_1(S_{F_a}S^2;\Z)$, and this is enough to prove this proposition. Note that half of the grading in \eqref{gradingnull} yields the map
$$f_a \colon (m_1,m_2) \mapsto  \frac{m_1+m_2}{2} - \frac{m_1^2}{4} + \frac{m_1 m_2}{2} - \frac{m_2^2}{4} + \sum_{k=1}^{m_1} \left \lfloor \frac{k}{1+a} \right \rfloor + \sum_{k=1}^{m_2} \left \lfloor \frac{k}{1-a} \right \rfloor,$$
for $m_1,m_2$ nonnegative integers and such that $m_1+m_2$ is even. Note that
$$\left \lfloor \frac{k}{1+a} \right \rfloor = \left \lfloor k - \frac{ka}{1+a} \right \rfloor = k-\left \lfloor \frac{ka}{1+a} \right \rfloor -1,$$
and similarly
$$\left \lfloor \frac{k}{1-a} \right \rfloor = \left \lfloor k + \frac{ka}{1-a} \right \rfloor = k+\left \lfloor \frac{ka}{1-a} \right \rfloor.$$
Then, we compute
\begin{eqnarray*}
f_a(m_1,m_2) &=&  \frac{m_1+m_2}{2} - \frac{m_1^2}{4} + \frac{m_1 m_2}{2} - \frac{m_2^2}{4} + \frac{m_1^2+m_1}{2} - m_1 - \sum_{k=1}^{m_1} \left \lfloor \frac{ka}{1+a} \right \rfloor \\ &+& \frac{m_2^2+m_2}{2} + \sum_{k=1}^{m_2} \left \lfloor \frac{ka}{1-a} \right \rfloor \\ &=& \left(\frac{m_1+m_2}{2}\right)^2 + m_2 - \sum_{k=1}^{m_1} \left \lfloor \frac{ka}{1+a} \right \rfloor + \sum_{k=1}^{m_2} \left \lfloor \frac{ka}{1-a} \right \rfloor.
\end{eqnarray*}
Let $n:= (m_1+m_2)/2$ and $m:=m_2$. Under this transformation, we have
\begin{eqnarray}\label{fbijekatok}
f_a\colon D:=\{(n,m) \in \Z^2_{\geq 0} \mid m \leq 2n\} &\to& \Z_{\geq 0} \\ (n,m) &\mapsto& n^2 + m - \sum_{k=1}^{2n-m} \left \lfloor \frac{ka}{1+a} \right \rfloor + \sum_{k=1}^{m} \left \lfloor \frac{ka}{1-a} \right \rfloor. \nonumber
\end{eqnarray}
\newline
\noindent\textbf{Claim}: Let $(n,m) \in D \subset \R^2$. Then $f_a(n,m)+1$ is the number of lattice points in $D$ below the line of slope $-(1-a)/a$ passing through the point $(n,m)$.
\begin{proof}[Proof of the Claim]
Let $D_a(n,m) \subset D$ be the subset consisting in points in $D$ lying below the line of slope $-(1-a)/a$ passing through the point $(n,m)$. The number of lattice points in $D_a(n,m)$ can be computed by
\begin{equation}\label{latticepts}
\mathcal{L}(D_a(n,m)) = \mathcal{L}(T_1) + \mathcal{L}(T_2) - \mathcal{L}(T_3) + \mathcal{L}(S),
\end{equation}
where $\mathcal{L}(A) = \# (A \cap \Z^2)$ denotes the number of lattice points in a subset $A \subset \R^2$, and the subsets $T_i \subset \R^2$, for $i=1,2,3$ and $S \subset \R^2$ are defined as follows. The subset $S$ is the line segment from $(n,0)$ to $(n,m)$. The triangle $T_1$ is delimited by the $x$-axis, the line $y=2x$ and (not including) the line $x=n$. The triangle $T_2$ is delimited by (not including) the line segment $S$, the line of slope $-(1-a)/a$ passing through the point $(n,m)$ and the $x$-axis. Finally, $T_3$ is the triangle delimited by the line $y=2x$, the line of slope $-(1-a)/a$ passing through the point $(n,m)$ and (not including) the line $x=n$, see Figure \ref{figlattice}.

\begin{figure}[H]
\centering
\includegraphics[scale=0.6]{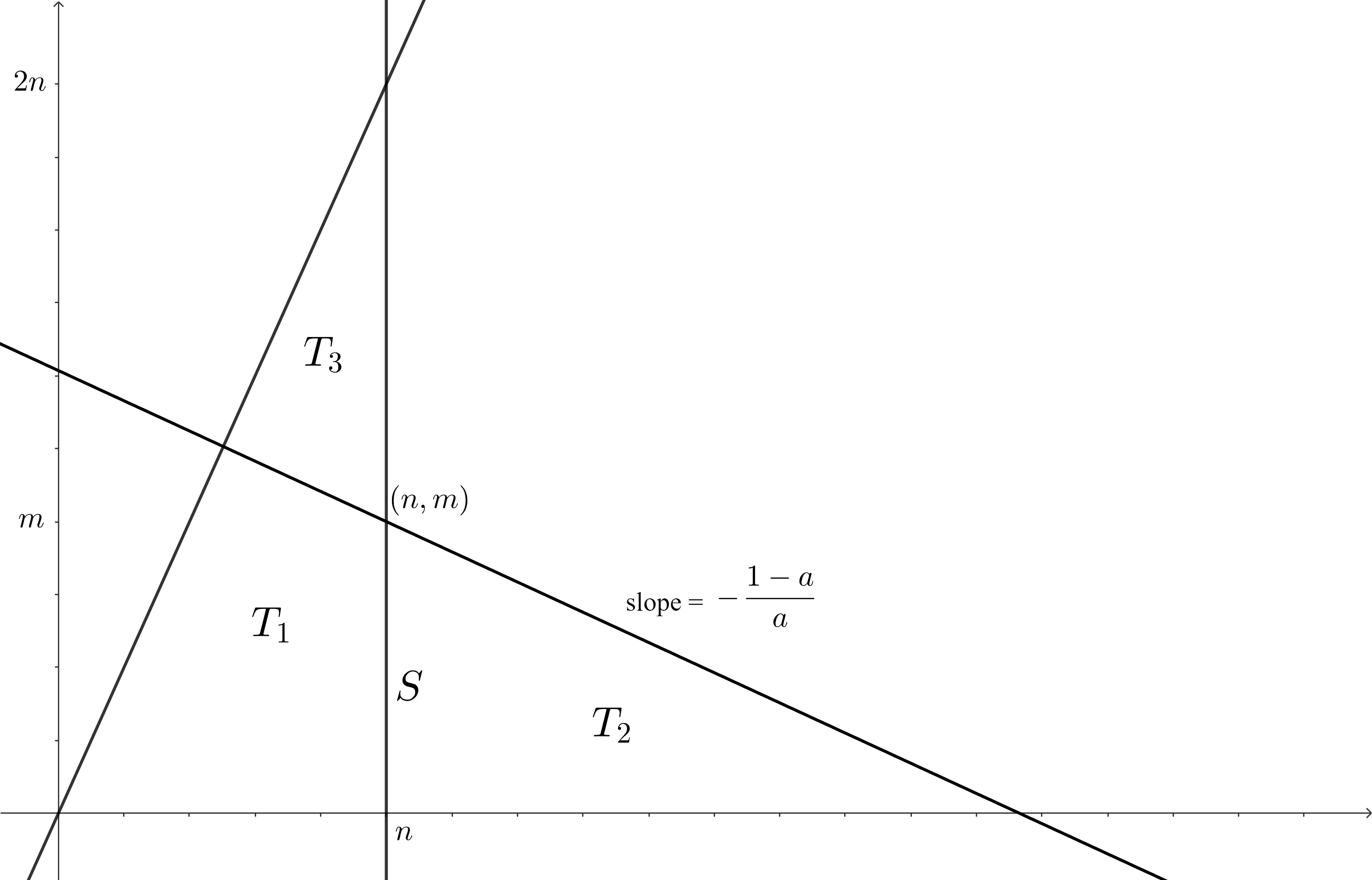}
\caption{Subsets $T_i \in \R^2$ and $S \in \R^2$.}\label{figlattice}
\end{figure}

Now we compute the number of lattice points in each of these sets. It is simple to check that
\begin{eqnarray}
\mathcal{L}(T_1) &=& 1+3+\ldots+2n-1 = n^2 \label{lattice1} \\ \mathcal{L}(S) &=& m+1 \label{latticeS} \\ \mathcal{L}(T_2) &=& \sum_{k=1}^m \left \lfloor k \frac{1}{\frac{1-a}{a}} \right \rfloor = \sum_{k=1}^m \left \lfloor \frac{ka}{1-a} \right \rfloor \label{lattice2}.
\end{eqnarray}
To compute $\mathcal{L}(T_3)$, we first apply the $SL(2,\Z)$ transformation
$\begin{bmatrix}
-1 & 0 \\ 2 & -1
\end{bmatrix}$
and then the translation by $(n,0)$ to the triangle $T_3$. So $\mathcal{L}(T_3)$ agrees with the number of lattice points in the new triangle delimited by the coordinate axis and the line with slope $-(1+a)/a$ passing through $(0,2n-m)$. Hence,
\begin{equation} \label{lattice3}
\mathcal{L}(T_3) = \sum_{k=1}^{2n-m} \left \lfloor \frac{ka}{1+a} \right \rfloor
\end{equation}
and putting \eqref{latticepts}, \eqref{lattice1}, \eqref{latticeS}, \eqref{lattice2} and \eqref{lattice3} together, we obtain $\mathcal{L}(D_a(n,m)) = f(n,m) + 1$, proving the claim.
\end{proof}
Since $a$ is irrational, the slope $-(1-a)/a$ is also irrational and hence, for each $j \in \Z_{\geq 1}$ there exists a unique $(n,m)$ such that $j=\mathcal{L}(D_a(n,m))$. Therefore, the fact $\mathcal{L}(D_a(n,m)) = f_a(n,m) + 1$ ensures that $f_a$ defined in \eqref{fbijekatok} is a bijection. Thus, the grading in \eqref{gradingnull} factors through the bijections
$$\{\gamma_{p_1}^{m_1}\gamma_{p_2}^{m_2}\mid m_1+m_2 \in 2\Z_{\geq 0}\} \xrightarrow[m=m_2]{n=(m_1+m_2)/2} \{(n,m) \in \Z^2_{\geq 0}\mid m\leq 2n \} \xrightarrow{2f} 2\Z_{\geq 0}.$$
This concludes the proof for $\Gamma = 0 \in H_1(S_{F_a}S^2;\Z)$. One can deal with the case $\Gamma = \overline{1}$ analogously.
\end{proof}
\end{prop}

Since the differential vanishes in the chain complex $ECC_*(S_{F_a}S^2,\lambda_{F_a},\Gamma)$, Proposition \ref{echkatok} gives also the computation of the ECH groups and, using the invariance due to Taubes, this recovers Theorem \ref{echrp3}. Now we are ready to compute the ECH spectrum of $(S_{F_a}S^2,\lambda_{F_a})$, proving Theorem \ref{katokspectrum}.

\begin{proof}[Proof of Theorem \ref{katokspectrum}]
By Proposition \ref{echkatok}, for any integer $k>0$,
$$ECH_{2k}(S_{F_a}S^2,\lambda_{F_a},0) \cong \Z_2$$
has exactly one generator. Let $\zeta_k$ be this generator. Since the $U$ map does not depend on the contact form by Theorem \ref{taubes}, then Theorem \ref{Umap} still holds in this case, i.e., $U\zeta_k = \zeta_{k-1}$, for $k\geq 1$. Moreover, a generator is an orbit set $\alpha = \gamma_1^{m_1}\gamma_2^{m_2}$, where $m_1+m_2$ is even, and as such, it has total action given by
$$\mathcal{A}(\alpha) = m_1\frac{2\pi}{1+a}+m_2\frac{2\pi}{1-a}.$$
Thus, the result follows from the fact that the $U$ map decreases the action.
\end{proof}

As a consequence, we recover the ECH spectrum for the unit cotangent bundle of the round sphere first computed in \cite[Theorem 1.4]{ferreira2021symplectic}, as follows. We have the following strict contactomorphisms
$$(S_{F_a}S^2,\lambda_{F_a}) \cong (S^*_{F_a}S^2,\lambda_{taut}) \cong (S^*S^2,\mathrm{f}_a\lambda_{taut}),$$
where $\mathrm{f}_a\lambda_{taut}$ is the contact form on the unit cotangent bundle of the round sphere $S^*S^2$ corresponding to the restriction of the tautological form on $S^*_{F_a}S^2$. Therefore, by definition of the ECH spectrum for the degenerate case in \eqref{cklim}, we obtain
\begin{align*}
c_k(S^*S^2,\lambda_{taut}) &= \lim_{n\to \infty} c_k(S^*S^2,\mathrm{f}_{a_n}\lambda_{taut}) \\ &= \lim_{n\to \infty}c_k(S_{F_{a_n}}S^2,\lambda_{F_{a_n}}) \\ &= (M_2\left(N(2\pi,2\pi)\right))_k,
\end{align*}
where we used Theorem \ref{katokspectrum} for a sequence $a_n$ of irrational numbers in $(0,1)$ converging to $0$. The same argument yields the conclusion that Theorem \ref{katokspectrum} holds for any real number $a \in [0,1)$.

\bibliographystyle{alpha}
\bibliography{biblio2}
\end{document}